\documentclass[11pt,letterpaper]{amsart}

\usepackage{amsmath, amscd, amssymb}
\usepackage[frame,cmtip,arrow,matrix,line,graph,curve]{xy}
\usepackage{graphpap, color}
\usepackage[mathscr]{eucal}
\usepackage{cancel}
\usepackage{verbatim}
\usepackage{adjustbox}
\usepackage{tikz}
\usepackage{float}
\usepackage{booktabs}
\usepackage{multirow}
\usepackage{tabularx}
\usepackage{bbm}
\usepackage{stmaryrd}
\usepackage{mathtools}

\usepackage{hyperref}
\usepackage{listings}

\numberwithin{equation}{section}

\newtheorem{theorem}{Theorem}[section]
\newtheorem{corollary}[theorem]{Corollary}
\newtheorem{proposition}[theorem]{Proposition}
\newtheorem{lemma}[theorem]{Lemma}

\newtheorem{conjecture}[theorem]{Conjecture}

\theoremstyle{definition}
\newtheorem{definition}[theorem]{Definition}
\newtheorem{question}[theorem]{Question}
\newtheorem{example}[theorem]{Example}
\newtheorem{remark}[theorem]{Remark}

\newcommand{\Z}{\mathbb{Z}}
\newcommand{\Q}{\mathbb{Q}}

\newcommand{\C}{\mathbb{C}}

\newcommand{\PP}{\mathbb{P}}

\newcommand{\chrom}{\mathbf{X}}

\def\CC{\mathbb{C}}

\def\FF{\mathbb{F}}

\def\QQ{\mathbb{Q}}

\def\RR{\mathbb{R}}

\def\ZZ{\mathbb{Z}}

\def\sO{{\mathscr O}}

\def\sK{\mathscr{K}}

\def\sT{\mathscr{T}}

\newcommand{\cal}{\mathcal}

\def\cM{{\cal M}}

\def\sO{{\cal O}}

\def\chrom{\mathscr{X}}

\def\hbar{\overline{h}}

\def\Mbar{\overline{\cM}}

\def\GL{\mathrm{GL} }
\def\SL{\mathrm{SL} }

\DeclareMathOperator{\Aut}{Aut}

\DeclareMathOperator{\Exp}{Exp}

\def\Ind{\mathrm{Ind} }

\def\Stab{\mathrm{Stab} }
\def\dim{\mathrm{dim} }

\def\log{\mathrm{log} }

\def\ch{\mathrm{ch} }

\def\multiset#1#2{\ensuremath{\left(\kern-.3em\left(\genfrac{}{}{0pt}{}{#1}{#2}\right)\kern-.3em\right)}}
\def\multisetBig#1#2{\ensuremath{\Big(\kern-.3em\Big(\genfrac{}{}{0pt}{}{#1}{#2}\Big)\kern-.3em\Big)}}
\def\multisetbody#1#2{\ensuremath{\big(\kern-.3em\big(\genfrac{}{}{0pt}{}{#1}{#2}\big)\kern-.3em\big)}}

\def\and{\quad{\rm and}\quad}
\def\lra{\longrightarrow }
\def\mapright#1{\,\smash{\mathop{\lra}\limits^{#1}}\,}

\def\beq{\begin{equation}}
\def\eeq{\end{equation}}
\def\ben{\begin{enumerate}}
\def\een{\end{enumerate}}

\def\and{\quad\text{and}\quad}

\def\a{\alpha}
\def\d{\delta}

\def\w{\omega}

\def\sfs{\mathsf{s}}
\def\sfp{\mathsf{p}}
\def\sfh{\mathsf{h}}
\def\sfe{\mathsf{e}}

\def\stan{\mathscr{S}}
\def\stanomega{\stan^{\omega}}
\def\charP{\mathscr{P}}
\def\charQ{\mathscr{Q}}
\def\symS{\mathbb{S}}
\def\csf{\mathsf{csf}}
\def\cqsf{\mathsf{cqsf}}
\def\sgn{\mathsf{sgn}}
\def\Par{\mathrm{Par}}
\def\centz{z}

\title{Characteristic polynomial of $\overline{\mathcal{M}}_{0,n}$ and log-concavity}
\date{October 12, 2025}

\author{Jinwon Choi}
\address{Department of Mathematics and Research Institute of Natural Sciences, Sookmyung Women's University, Seoul 04310, Korea}
\email{jwchoi@sookmyung.ac.kr}

\author{Young-Hoon Kiem}
\address{School of Mathematics, Korea Institute for Advanced Study, 85 Hoegiro, Dongdaemun-gu, Seoul 02455, Korea}
\email{kiem@kias.re.kr}

\author{Donggun Lee}
\address{Center for Complex Geometry, Institute for Basic Science (IBS), 55 Expo-ro, Yuseong-gu, Daejeon 34126, Korea}
\email{dglee@ibs.re.kr}

\makeatletter
\@namedef{subjclassname@2020}{\textup{2020} Mathematics Subject Classification}
\makeatother
\thanks{JC was partially supported by NRF grant 2018R1C1B6005600. 
DL was supported by the Institute for Basic Science IBS-R032-D1 and Korea NRF grant 2021R1A2C1093787.}

\begin{document}

\begin{abstract}
Motivated by Stanley's generalization of the chromatic polynomial of a graph to the chromatic symmetric function, we introduce the characteristic polynomial of a representation of the symmetric group, or more generally, of a symmetric function. When the representation arises from geometry, the coefficients of its characteristic polynomial tend to form a log-concave sequence. To illustrate, we investigate explicit examples, including the $n$-fold products of the projective spaces, the GIT moduli spaces of points on $\mathbb{P}^1$ and Hessenberg varieties. 
Our main focus lies on the cohomology of the moduli space of pointed rational curves, for which we prove asymptotic formulas of its characteristic polynomial and establish asymptotic log-concavity.  
\end{abstract}

\maketitle
\tableofcontents

\section{Introduction}

In \cite{Sta2}, Stanley introduced the chromatic symmetric function of a graph, as a symmetric function generalization of the chromatic polynomial of a graph. In this paper, by taking the inverse approach, we associate a polynomial, called the \emph{characteristic polynomial}, to each symmetric function. Although the polynomials we consider are not the chromatic polynomials of graphs in general, those obtained from geometric contexts appear to share one of the most interesting features of chromatic polynomials, namely the \emph{log-concavity} (cf.~\S\ref{S2}).

\subsection{Chromatic polynomial and characteristic polynomial} 
The \emph{chromatic polynomial} of a graph $\Gamma$ is 
a polynomial $\mathscr{X}_\Gamma(m)\in \Q[m]$ such that for each $m_0\in \mathbb{N}=\Z_{\geq1}$, the value $\mathscr{X}_\Gamma(m_0)$ equals the number of proper colorings of the vertices of $\Gamma$ using $m_0$ colors. Here, a coloring is called \emph{proper} if any two adjacent vertices (that is, vertices connected by an edge) take different colors. In this paper, we regard $\mathscr{X}_\Gamma(m)$ as a polynomial in the variable $m$, and write $\mathscr{X}_\Gamma(m_0)$ for its evaluation at a positive integer $m_0$.

The \emph{chromatic symmetric function} $\csf_\Gamma$ of  $\Gamma$ is defined by
\[\csf_{\Gamma}=\sum_{\substack{\kappa:V(\Gamma)\to \Z_{\geq1},\\\mathrm{proper}}}
\prod_{v\in V(\Gamma) }x_{\kappa(v)}\]
where $\{x_i\}_{i\in \Z_{\geq 1}}$ are formal variables and $V(\Gamma)$ is the set of vertices of $\Gamma$. This is obviously a symmetric function in $x_i$ and the substitution
\beq \label{4}x_1=x_2=\cdots=x_{m_0}=1 \and x_{m_0+1}=x_{m_0+2}=\cdots =0\eeq
to $\csf_\Gamma$ gives the integer $\mathscr{X}_\Gamma(m_0)$. In this sense, $\csf_{\Gamma}$ can be considered as a generalization of $\mathscr{X}_\Gamma$. However, the log-concavity proved in \cite{Huh} appears only after the substitution \eqref{4} and hence it makes sense to extend the chromatic polynomials to more general contexts when one is interested in the log-concavity. 

In this paper, we first observe that the process of producing a polynomial out of a symmetric function by the substitution \eqref{4} works not only for chromatic symmetric functions of graphs but also for all symmetric functions. 
Indeed, the substitution \eqref{4} produces a map $\stan_F:\mathbb{N}\to \QQ$ for any symmetric function $F$ in the variables $x_i$ with rational coefficients. 
It is straightforward to check that $\stan_F$ is a polynomial function whose associated polynomial in $\QQ[m]$ is called the \emph{characteristic polynomial} $\stan_F(m)$ of $F$ (Definition~\ref{def:charpoly}).

\smallskip

It turns out that this definition is quite natural. To explain this, we first write it in a different but equivalent form. 
Let $\Lambda$ be the power series ring of the symmetric functions in $x_i$ over $\Q$, i.e.
\[\Lambda=\lim_{\longleftarrow}\Q\llbracket x_1,\cdots, x_n\rrbracket^{\symS_n},\]
and let $\Lambda_n\subset \Lambda$ be the subspace of elements of degree $n$.
Let $\sfp_n=\sum_{i\geq 1}x_i^n$ be the $n$th power sum symmetric function for $n\geq 1$ and let $\sfp_0=1$. It is well-known that $\Lambda\cong \Q\llbracket \sfp_1,\sfp_2,\cdots \rrbracket$. The map associating the characteristic polynomial to each symmetric function
\[\widehat\stan:\Lambda \lra \Q[m]\llbracket q\rrbracket, \quad (F\in \Lambda_n)\mapsto \stan_F(m)q^n\]
is precisely the $\Q$-algebra homomorphism 
which sends $\sfp_n$ to $mq^n$ for every $n\geq 1$, where $q$ is a formal variable. 

A remarkable feature of this map is that it behaves nicely with \emph{plethysm} $\circ$ on $\Lambda$ (see~\S\ref{ss:plethysm}). 
In fact, it converts the plethysm $F\circ G$ of symmetric functions $F\in \Lambda_n$ and $G\in \Lambda_r$ 
to the composition $(\stan_F\circ \stan_G)(m)=\stan_F(\stan_G(m))$ 
(Proposition~\ref{p:plethysm}). 
Conversely, this feature uniquely determines $\widehat\stan$ 
as a $\Q$-algebra homomorphism $\Lambda\to \Q[m]\llbracket q\rrbracket$ 
(Theorem~\ref{thm:characterize.chi}).

\smallskip

We can also define the characteristic polynomial of a symmetric group representation.
Let $\symS_n$ denote the symmetric group of $n$ letters and $R_n$ be the $\Q$-vector space spanned by irreducible $\symS_n$-representations. Let 
\[\ch=\ch_{\symS_n}:R_n\xrightarrow{~\cong~}\Lambda_n\]
be the Frobenius characteristic map (cf.~\eqref{eq:Frob}). The \emph{characteristic polynomial} of an $\symS_n$-module $V$ is defined by $\stan_V(m)=\stan_{\ch_{\symS_n}(V)}(m)$.

\subsection{Characteristic polynomials of varieties}
In this paper, we work out the characteristic polynomials of 
examples from geometry, especially the cohomology of varieties
equipped with an action of $\symS_n$. 
These include the $n$-fold products of the projective spaces 
(cf.~\S\ref{S4.1}) and the GIT moduli space of $n$ points on $\PP^1$ for odd $n$ (cf.~\S\ref{S4.2}). In both cases, the $\symS_n$-action that permutes the $n$ ordered points induces graded $\symS_n$-representations on the cohomology. We compute the characteristic polynomials of these representations and establish suitable log-concavity.

Another remarkable example is a regular semisimple \emph{Hessenberg variety} $X_h$ introduced in \cite{DMPS}. It is a subvariety in the flag variety of $\C^n$, defined by a function $h:\{1,\cdots, n\}\to \{1,\cdots, n\}$ satisfying $h(i)\ge i$ and $h(i)\ge h(i-1)$ for all $i>1$.
There is a natural action of $\symS_n$ on its cohomology $H(X_h)$ by \cite{Tym} whose characteristic polynomial satisfies
\beq\label{y40} \omega\ch_{\symS_n} H(X_h)=\csf_{\Gamma_h}\eeq
where $\omega$ is the involution on $\Lambda$ that sends $\sfp_n$ to $(-1)^n\sfp_n$ and $\Gamma_h$ is 
the incomparability graph associated to $h$ (cf.~Theorem~\ref{thm:SWconj}). While this equality can be written in a graded version, we stick to this simpler form here. 
In particular, applying $\stan$ to \eqref{y40}, we get  
\beq\label{y41} (-1)^n\stan_{H(X_h)}(-m)=\chrom_{\Gamma_h}(m).\eeq
By \cite{Huh}, the characteristic polynomial $\stan_{H(X_h)}(m)$ of the Hessenberg variety is log-concave. See \S\ref{ss:Hess} for further details. 

\subsection{Characteristic polynomial of $\Mbar_{0,n}$}
The main focus of this paper lies on the moduli space $\Mbar_{0,n}$ of stable rational curves with $n$ marked points.
See \cite{CKL, CKL2} for basic facts about this moduli space and its cohomology. 
This is a case where the characteristic polynomial is not the chromatic polynomial of any graph.
The $\symS_n$-action on $\Mbar_{0,n}$ permuting the $n$ marked points gives rise to a graded $\symS_n$-representation on its cohomology $H^*(\Mbar_{0,n})$. Consequently,
we obtain a bivariate polynomial 
\[\charP_n(m,t)=\sum_{k=0}^{n-3}\charP_{n,k}(m)t^k ~\in \Q[m,t], \quad \text{ where }~\charP_{n,k}(m)=\stan_{\ch_{\symS_n}H^{2k}(\Mbar_{0,n})}(m)\]
for each $k\geq0$.
We write the generating series of $\charP_n$ as 
\[\charP=1+mq+\frac{m(m+1)}{2}q^2+\sum_{n\geq 3}\charP_n(m,t)q^n~\in \Q[m,t]\llbracket q\rrbracket\]
for a formal parameter $q$ recording $n$. 

\smallskip

Recursive formulas for computing $\charP_n$ can be deduced from \cite{CKL,CKL2}, where we developed recursive formulas for $\ch_{\symS_n}H^*(\Mbar_{0,n})$ that use only arithmetic operations and plethysm. See \cite[Theorem~1.3]{CKL2} or \S\ref{ss:recursion.rep} and \ref{ss:recursion.char.poly}.

To present explicit formulas, let $\charQ_{n,k}$ denote the characteristic polynomial 
\[\charQ_{n,k}(m)=\stan_{\ch_{\symS_n}H^{2k}(\Mbar_{0,n+1})}(m)\]
of the $\symS_n$-representation on $H^{2k}(\Mbar_{0,n+1})$ induced by the action that permutes the first $n$ markings while keeping the last fixed. The advantage of $\charQ_{n,k}$ is that there is a rather straightforward combinatorial formula 
	\[\ch_{\symS_n}H^{2k}(\Mbar_{0,n+1})=\sum_{T\in \sT_{n,k}}\ch_{\symS_n}(U_T)\]
for the Frobenius characteristic  
in terms of \emph{weighted rooted trees} $T$ and the associated permutation representations $U_T$ by \cite[Proposition~5.12]{CKL}. In particular, we have 
\beq\label{6}\charQ_{n,k}=\sum_{T\in \sT_{n,k}}\stan_{\ch_{\symS_n}(U_T)}.\eeq
For precise definitions of $T$, $\sT_{n,k}$ and $U_T$, we refer the reader to \S\ref{s:M0n}.

We define the generating series of $\charQ_{n,k}$ by
\[\charQ=1+mq+\sum_{n\geq 2}\charQ_n(m,t)q^n ~\in \Q[m,t]\llbracket  q\rrbracket, \quad \text{ where }~\charQ_n(m,t)=\sum_{k=0}^{n-2}\charQ_{n,k}(m)t^k.\]
Moreover, we consider a partial sum 
\[\charQ_{n,k}^+=\sum_{T\in \sT_{n,k}^+}\stan_{\ch_{\symS_n}(U_T)}\] of \eqref{6}
where $\sT_{n,k}^+\subset \sT_{n,k}$ consists of weighted rooted trees having positive weights at the root vertex. The generating series for $\charQ_{n,k}^+$ is defined as
\[\charQ^+=mq+\sum_{n\geq 2}\charQ_n^+q^n ~\in \Q[m,t]\llbracket  q\rrbracket, \quad \text{ where }~\charQ_n^+=\sum_{k=0}^{n-2}\charQ_{n,k}^+t^k.\]
Then $\charP$ can be computed from $\charQ$, $\charQ$ can be computed from $\charQ^+$ and $\charQ^+$ can be computed by its own recursion as follows (cf.~\S\ref{s:M0n}).
\begin{theorem}[Theorem~\ref{thm:recursion}]
$\charP$ is determined by the following.
\begin{enumerate}
	\item $\charP$ and $\charQ$ are related by the formula
	\[
	(1+t)\charP=(1+t+mqt)\charQ- \frac{t}{2}(\charQ^2-\charQ^{[2]}), 
	\]
    where $\charQ^{[2]}(m,q,t)=\charQ(m,q^2,t^2)$.
	\item $\charQ$ is the plethystic exponential of $\charQ^+$: $\charQ=\Exp(\charQ^+)$.
	\item $\charQ^+$ satisfies the following two equivalent formulas.
	\[\begin{split}
		\mathrm{(Recursive)} \quad &\charQ^+=mq +\sum_{r\geq 3}\left(\sum_{i=1}^{r-2}t^i\right)(\sfh_r\circ \charQ^+),\\
		\mathrm{(Exponential)} \quad &\Exp(t\charQ^+)=t^2\Exp(\charQ^+)+(1-t)(1+t+mqt ).
	\end{split}\]
\end{enumerate}
\end{theorem}

An equivalent polynomial version of this theorem is the following.

\begin{corollary}[Corollary~\ref{cor:recursion}]\label{cor:intro.recursion}
	For $n\geq 3$, $\charP_n$ and $\charQ_n$ are related by
	\[(1+t)\charP_n=\charQ_n-\frac{1}{2}t\left(\sum_{h=2}^{n-2}\charQ_h\charQ_{n-h}-\charQ_{\frac{n}{2}}^{[2]}
    \right),\]
	where $\charQ_\frac{n}{2}=0$ for $n$ odd, and $\charQ_{\frac{n}{2}}^{[2]}(m,t)=\charQ_{\frac{n}{2}}(m,t^2)$.
	
	Set $\charQ_1^+=m$. For $n\geq2$, $\charQ_n^+$ and $\charQ_n$ satisfy 
	\[\begin{split}
	    \charQ_n^+&=\sum_{\lambda\vdash n}\left(\sum_{i=1}^{\ell(\lambda)-2}t^i\right)\prod_{j=1}^m \left(\sfh_{r_j}\circ\charQ_{n_j}^+\right) \and
		\charQ_n=\sum_{\lambda\vdash n}\prod_{j=1}^m \left(\sfh_{r_j}\circ\charQ_{n_j}^+\right)
	\end{split}\]
	where $n_j$ denote the parts of $\lambda$ with multiplicities $r_j$ so that $\lambda=(n_1^{r_1},\cdots, n_s^{r_s})$ with $n_1>\cdots >n_s>0$
 and $\ell(\lambda)=\sum_{j=1}^sr_j$.
\end{corollary}

By using the above formulas, we implemented a computer program to explicitly compute the characteristic polynomials, and carried out computations for $n\le 72$. The computational results allow us to investigate the log-concavity properties of the characteristic polynomials. Since $\charP_n$ and $\charQ_n$ are bivariate polynomials, various types of log-concavity can be considered (cf.~\S\ref{S2}), and our computations indicate that they exhibit several such properties. Based on this numerical evidence, we propose a conjecture concerning the log-concavity of the characteristic polynomials (Conjecture~\ref{conj}).

\subsection{Asymptotic formulas and log-concavity}

We study asymptotic formulas of $\charP_{n,k}(m)$, as $n$ grows, in two directions: 

\begin{enumerate}
	\item the values $\charP_{n,k}(m_0)$ at a given $m_0\geq 1$ for given $k\geq0$;
	\item the coefficients of $m^{n-j}$ in $\charP_{n,k}$, denoted by $\charP_{n,k}^{n-j}$, for given $j,k\geq0$.
\end{enumerate}
These formulas lead us to corresponding asymptotic log-concavity properties which are far reaching generalizations of the corresponding results in \cite{ACM, CKL2}.

For the asymptotics, we use the constants $$c_k=\frac{(k+1)^{k-1}}{k!} \and d_k=\frac{(k+1)^{k-2}}{k!}.$$ 
\begin{theorem}[Theorem~\ref{thm:asymp.value}]\label{thm:intro.value}
	Fix $m_0\geq 1$ and $k\geq0$. Then as $n$ grows,
	\[\begin{split}
		&\charQ_{n,k}(m_0)=\frac{c_k}{((k+1)m_0-1)!}n^{(k+1)m_0-1}+o(n^{(k+1)m_0-1}) ~\text{ and}\\
		&\charP_{n,k}(m_0)=\frac{d_k}{((k+1)m_0-1)!}n^{(k+1)m_0-1}+o(n^{(k+1)m_0-1}).
	\end{split} \]
\end{theorem}
From this, the following asymptotic log-concavity follows.
\begin{corollary}[Corollary~\ref{cor:asymp.lc.value}]\label{cor:intro.value}
	Fix $m_0,k\geq 1$. Then, 
	\[\begin{split}
		&\charP_{n,k}(m_0)^2\geq \charP_{n,k-1}(m_0)\charP_{n,k+1}(m_0) ~\text{ and}\\
		&\charQ_{n,k}(m_0)^2\geq \charQ_{n,k-1}(m_0)\charQ_{n,k+1}(m_0)
	\end{split}
	\]
	for sufficiently large $n$. Moreover, when $m_0>1$, the inequalities are strict.
\end{corollary}
Theorem~\ref{thm:intro.value} and Corollary~\ref{cor:intro.value}  
generalize the asymptotic formulas and the asymptotic log-concavity for the Betti numbers of $\Mbar_{0,n}/\symS_n$ in \cite[\S6]{CKL2} which correspond to the case of $m_0=1$ (cf. Remark \ref{re:19}). 

\smallskip

Next, we investigate the asymptotic behaviors of the coefficients in $\charP_{n,k}$ and $\charQ_{n,k}$. Recall that the \emph{signless Stirling number of the first kind} $c(n,j)$ is the number of $\sigma\in \symS_n$ with $j$ disjoint cycles. 

Write 
$\charP_{n,k}(m)=\sum_j \charP^j_{n,k}m^j$ and $\charQ_{n,k}(m)=\sum_j\charQ^j_{n,k}m^j$.

\begin{theorem}[Theorem~\ref{thm:mu.coeff}]\label{thm:intro.coeff}
	Let $j,k\in \Z_{\geq 0}$. Then as $n$ grows,  
	\[\begin{split}
		&
        \charQ_{n,k}^{n-j}
        =c_k\cdot\frac{(k+1)^{n-j}c(n,n-j)}{n!}+o\left(\frac{(k+1)^nn^{2j}}{n!}\right) ~\text{ and}\\
		&
        \charP_{n,k}^{n-j}=d_k\cdot\frac{(k+1)^{n-j}c(n,n-j)}{n!}+o\left(\frac{(k+1)^nn^{2j}}{n!}\right).
	\end{split}
	\]
\end{theorem}
Based on this, we obtain the following two asymptotic log-concavity results about the coefficients when we vary the powers of $t$ and $m$, respectively. In fact, in the first case, they are asymptotically \emph{ultra} log-concave.
\begin{corollary}[Corollary~\ref{cor:lc.coeff.t}]\label{cor:intro.coeff.t}
	Fix $j\geq 0$ and $k\geq 1$. Then,
	\[\begin{split}
		&\left(\frac{\charP_{n,k}^{n-j}}{\binom{n-3}{k}}\right)^2\geq \left(\frac{\charP_{n,k-1}^{n-j}}{\binom{n-3}{k-1}}\right)\left(\frac{\charP_{n,k+1}^{n-j}}{\binom{n-3}{k+1}}\right) ~\text{ and}\\
		&
        \left(\frac{\charQ_{n,k}^{n-j}}{\binom{n-2}{k}}\right)^2\geq \left(\frac{\charQ_{n,k-1}^{n-j}}{\binom{n-2}{k-1}}\right)\left(\frac{\charQ_{n,k+1}^{n-j}}{\binom{n-2}{k+1}}\right)
	\end{split}
	\]
	for any sufficiently large $n$. 
\end{corollary}

\begin{corollary}[Corollary~\ref{cor:lc.coeff.m}]\label{cor:intro.coeff.m}
	Fix $j\geq 1$ and $k\geq 0$. Then,
	\[\begin{split}
		&\left(\charP_{n,k}^{n-j}\right)^2\geq \left(\charP_{n,k}^{n-j-1}\right)\cdot\left(\charP_{n,k}^{n-j+1}\right) ~\text{ and}\\
		&\left(\charQ_{n,k}^{n-j}\right)^2\geq \left(\charQ_{n,k}^{n-j-1}\right)\cdot\left(\charQ_{n,k}^{n-j+1}\right)
	\end{split}
	\]
	for any sufficiently large $n$.
\end{corollary}

Theorem~\ref{thm:intro.coeff} and Corollary~\ref{cor:intro.coeff.t} generalize the results by Aluffi-Chen-Marcolli in \cite{ACM} on 
the Betti numbers of $\Mbar_{0,n+1}$, which correspond to the case where $j=0$ (Remark~\ref{rem:leading.coeff}).
It remains open whether the ultra log-concavity for~$j=0$ in Corollary~\ref{cor:intro.coeff.t} holds for all~$n$, whereas explicit formulas for the Betti numbers have recently been obtained in~\cite{AMN,EFMPV}.

\medskip

All cohomology groups in this paper are singular cohomology with rational coefficients and all $\symS_n$-representations are defined over $\QQ$.

\medskip

\noindent\textbf{Acknowledgement.} 
Part of this paper was written during the third author's visit to IAS in the fall of 2024. He is grateful to June Huh for the invitation and to IAS for its hospitality.

\bigskip

\section{Two-dimensional log-concavity}\label{S2}
A sequence $\{a_k\}$ of nonnegative real numbers is called \emph{log-concave} if 
\beq\label{y50}
a_k^2\ge a_{k-1}a_{k+1} \quad \text{for all }k.\eeq 
The log-concavity extends to arbitrary real sequences by taking the sequence of absolute values $|a_k|$. 
A nonnegative finite sequence $\left\{a_k\right\}_{k=0}^n$ is \emph{ultra log-concave} if the normalized sequence $$\left\{\frac{a_k}{\binom{n}{k}}\right\}_{0\le k\le n}$$ is log-concave. 
In what follows, any property defined in terms of log-concavity is similarly extended to define the corresponding notion of ultra log-concavity.
 
A polynomial $f(m)=\sum_{k=0}^na_km^k\in \RR_{\ge 0}[m]$ with nonnegative coefficients is said to be \emph{log-concave} if the sequence $\{a_k\}$ is log-concave. We say $f(m)$ has \emph{no internal zeros} if there do not exist integers $i<j<k$ such that $a_i\ne 0$, $a_j =0$ and $a_k\ne 0$. 

In this section, we generalize the notion of log-concavity to a two-dimensional array $\{A^j_k\}_{0\le j,k\le N}$ of nonnegative real numbers or to a bivariate polynomial
\beq\label{y0} A (m,t) =\sum_{k\ge 0}\sum_{j\ge 0 } A_k^j m^jt^k ~\in \RR_{\geq0}[m,t].\eeq
Let $A_k(m) = \sum_{j\ge 0} A_k^jm^j$ and  $A^j(t) = \sum_{k\ge 0} A_k^jt^k$ so that we have
$$A(m,t)=\sum_kA_k(m)t^k=\sum_jA^j(t)m^j.$$
From the matrix $\{A^j_k\}$, there are many ways to extract a sequence of numbers and define log-concavity.  In this paper, we will consider the following natural choices. 

\begin{definition}\label{def:lc} Let $A (m,t) =\sum_{j,k\ge 0}A_k^j m^jt^k ~\in \RR_{\geq0}[m,t]$ and $t_0, m_0\in \RR$.
    \begin{enumerate}
        \item $A(m,t)$ is \emph{length log-concave in degree $k$} if $A_k(m) = \sum_{j\ge 0} A_k^jm^j\in \RR_{\geq0}[m]$ is log-concave.
        \item $A(m,t)$ is \emph{length log-concave at $t=t_0$} if 
        $A(m,t_0)\in \RR_{\geq0}[m]$ is log-concave. 
        \item $A(m,t)$ is \emph{degree log-concave in length $j$} if $A^j(t)= \sum_{k\ge 0} A_k^jt^k\in \RR_{\geq0}[t]$ is log-concave. 
        \item $A(m,t)$ is \emph{degree log-concave at $m=m_0$} 
        if 
        $A(m_0,t)\in \RR_{\geq0}[t]$ is log-concave.
    \end{enumerate}
\end{definition}
This definition extends to arbitrary bivariate real polynomials $A(m,t)\in \RR[m,t]$ by taking the absolute values $|A^j_k|$ of the coefficients. 

We call $t$ the \emph{degree} variable and $m$ the \emph{length} variable since
we will see below that if $A(m,t)$ comes from the cohomology $H^*(X)$ of a space $X$ acted on by the symmetric group $\symS_n$, the exponent of the variable $m$ corresponds to the length of a partition and the exponent of the variable $t$ keeps track of the cohomological degree. 
 
\begin{remark}\label{y6}
In \S\ref{S3.3}, we will associate a polynomial $\stan_V(m)\in \Q[m]$ of degree $n$ 
to each finite dimensional $\symS_n$-representation $V\ne 0$ by applying the principal specialization map in \cite[\S7.8]{Stabook} to the Frobenius characteristic of $V$. 
By Lemma~\ref{l:mult}, the coefficient of $m^n$ in $\stan_V(m)$ is $\frac{1}{n!}\dim\, V$ and the dimension of the invariant part equals $$\dim\, V^{\symS_n}=\stan_V(1).$$ 
When the $\symS_n$-representation $V=\bigoplus_{k\ge 0}V_k$ is graded, we associate the polynomial $$\stan_V(m,t)=\sum_k \stan_{V_k}(m)t^k\in \Q[m,t].$$ 
In this case, the degree log-concavity in length $n$ is equivalent to the 
log-concavity of the sequence $\{\dim\, V_k\}_k$ 
while the degree log-concavity at $m=1$ is equivalent to the 
log-concavity of the sequence $\{\dim\, V_k^{\symS_n}\}_k$ 
of the dimensions of the invariant parts. 

In particular, when $$V=\bigoplus_{k\ge 0}H^{2k}(\overline{\cM}_{0,n})$$ is the cohomology of the moduli space $\overline{\cM}_{0,n}$ of $n$-pointed stable curves of genus $0$, we note that  
the degree log-concavity in length $n$ 
is the log-concavity of the Betti numbers proved asymptotically by Aluffi-Chen-Marcolli \cite{ACM}, while 
the degree log-concavity at $m=1$ 
is the log-concavity of the invariant part proved asymptotically in \cite{CKL2}. 
Below we will find more log-concavity results about the graded $\symS_n$-module $H^*(\overline{\cM}_{0,n})$ in terms of the  log-concavities in Definition~\ref{def:lc}.
\end{remark}

\begin{remark}\label{y7}
In \S\ref{S3.2}, we will associate a polynomial $\stan_F(m)\in \Q[m]$ of degree at most $n$ to each  symmetric function $F\in \Lambda_n$ of degree $n$ by the principal specialization map in \cite[\S7.8]{Stabook}. 
When $F=\sum_{k\ge 0}F_kt^k\in \Lambda_n[t]$ with $F_k\in \Lambda_n$, we associate $\stan_F(m,t)=\sum_{k\ge 0}\stan_{F_k}(m)t^k\in \Q[m,t].$

Given a graph $\Gamma$, we have the chromatic quasisymmetric function $\cqsf_\Gamma(t)$ defined in \cite{SW}. 
When $\Gamma$ is the graph associated to a Hessenberg function in \S\ref{ss:Hess}, $F=\cqsf_\Gamma(t)\in\Lambda_n[t]$ and
$\stan_F(m,1)=\stan_{\csf_\Gamma}(m)$ is the usual chromatic polynomial $\chrom_\Gamma(m)$ of $\Gamma$. 
Therefore, the length log-concavity at $t=1$ of $\stan_F(m,t)$ is the log-concavity of the chromatic polynomial of the graph $\Gamma$ which is a special case of the celebrated theorem of June Huh in  \cite{Huh}. 
\end{remark}

We end this section with the following well known results which will be used below. 

\begin{lemma}\cite[Proposition 2]{Sta3}\label{lem:lc.product}
	If $f, g\in \RR_{\geq0}[m]$ are log-concave with no internal zeros, then so is $fg\in \RR_{\geq0}[m]$.
\end{lemma}

\begin{lemma}\cite[Theorem 2]{Sta3}\label{lem:realrooted}
    If $f\in \RR_{\geq0}[m]$ has only real roots, then $f$ is log-concave.
\end{lemma}

\bigskip

\section{Characteristic polynomials of symmetric functions and representations}\label{s:char.poly}

In this section, we investigate useful properties of the principal specialization map or the \emph{Stanley map} in \cite[\S7.8]{Stabook} 
which associates a polynomial $\stan_F(m)\in \Q[m]$ to each symmetric function $F\in \Lambda_n$ of degree $n$. 
A remarkable feature of the Stanley map 
$$\stan:\Lambda_n\lra \Q[m], \quad F\mapsto \stan_F(m)$$
is that it sends the plethysm operation on $\Lambda_n$ to the composition on $\Q[m]$ (Proposition~\ref{p:plethysm}). 
In fact, this property uniquely determines the Stanley map (Theorem~\ref{thm:characterize.chi}). 
Since the plethysm operation usually causes computational difficulties, the Stanley map levitates computational burdens while keeping core information,  just like the characteristic polynomial of a matrix. 
In this vein, we call $\stan_F(m)\in \Q[m]$ the \emph{characteristic polynomial} of $F$ (Definition~\ref{def:charpoly}). 
When the symmetric function is the chromatic symmetric function of a graph defined by Stanley in \cite{Sta2}, its characteristic polynomial is the usual chromatic polynomial of Birkhoff's, which is log-concave by \cite{Huh}. 

By applying the Stanley map to the Frobenius characteristic $\ch(V)\in \Lambda_n$ of an $\symS_n$-module $V$, we obtain a polynomial 
$\stan_V(m)\in \Q[m]$ which we call the \emph{characteristic polynomial} of $V$. When 
$V=\bigoplus_{k\ge 0} V_k$ is graded, its characteristic polynomial 
$$\stan_V(m,t)=\sum_k \stan_{V_k}(m) t^k\in \Q[m,t]$$
allows us to consider the two-dimensional log-concavities in Definition~\ref{def:lc}.
In the subsequent sections, we will find interesting examples where the log-concavities hold.

\subsection{Symmetric functions and plethysm}\label{ss:plethysm}
In this subsection, we set up our notation on symmetric functions and recall basic facts about plethysm. 
For further details and undefined terms, we refer the reader to \cite{Mac}. 

Let 
\[\Lambda=\lim_{\longleftarrow}\Q\llbracket x_1,\cdots,x_n\rrbracket^{\symS_n}\] be the power series ring of symmetric functions in variables $x_1,x_2,\cdots$, with \emph{rational} coefficients.
Let $\Lambda_n\subset \Lambda$ denote the subspace of elements of degree $n$, so that $\Lambda=\prod_{n\geq 0}\Lambda_n$.
For an integer $n\ge 1$, let 
\[\sfp_n=\sum_{i\geq 1}x_i^n, \quad \sfh_n=\sum_{1\le i_1 \le \cdots \le i_n}x_{i_1}\cdots x_{i_n} \and \sfe_n=\sum_{1\le i_1 < \cdots < i_n}x_{i_1}\cdots x_{i_n} \]
be the $n$-th power sum symmetric function, 
homogeneous symmetric function and 
elementary symmetric function, respectively. 
For convenience, we set $\sfp_0=\sfh_0=\sfe_0=1$. 
For a partition $\lambda=(\lambda_1,\cdots,\lambda_\ell)\vdash n$,  
let 
\[\sfp_\lambda=\sfp_{\lambda_1}\cdots\sfp_{\lambda_{\ell}}, \quad \sfh_\lambda=\sfh_{\lambda_1}\cdots\sfh_{\lambda_{\ell}} \and \sfe_\lambda=\sfe_{\lambda_1}\cdots\sfe_{\lambda_{\ell}}.\]
Then each of $\{\sfp_\lambda\}_{\lambda\vdash n}$, $\{\sfh_\lambda\}_{\lambda\vdash n}$ and $\{\sfe_\lambda\}_{\lambda\vdash n}$ is a basis for $\Lambda_n$.

For a partition $\lambda$, define $$\centz_\lambda = \prod_{i\geq 1}i^{m_i} m_i !$$ where $m_i$ is the number of parts of $\lambda$ equal to $i$. We let $\centz_0=1$. 
Define a symmetric bilinear form $\langle -,-\rangle$ on $\Lambda$  by 
\beq\label{y4}
\langle\sfp_\lambda,\sfp_\mu\rangle=\centz_\lambda\d_{\lambda \mu}\eeq 
so that $\{\sfp_\lambda\}$ is an orthogonal basis. 

There is a ring homomorphism $\omega:\Lambda \to \Lambda $ defined by 
\beq\label{y5} \omega(\sfp_n)=(-1)^{n-1} \sfp_n\eeq 
Clearly $\omega$ is an isometric involution and satisfies $\omega(\sfe_\lambda)=\sfh_\lambda$ for any $\lambda\vdash n$.

\medskip
On $\Lambda$, there is an associative binary operation $\circ$, called the \emph{plethysm}, 
which is uniquely determined by the following properties: for $F,G,H\in\Lambda$,
\begin{enumerate}
	\item $F\circ \sfp_n=\sfp_n\circ F=F(x_1^n,x_2^n,\cdots)$ for $n\geq 1$;
	\item $(F+G)\circ H=F\circ H+G\circ H$;
	\item $(FG)\circ H=(F\circ H)(G\circ H)$.
\end{enumerate}
Note that $\sfp_1$ is the two-sided identity with respect to $\circ$ by (1). This naturally extends to such an operation $\circ$ on $\Lambda\llbracket t\rrbracket$ satisfying (2) and (3) with $\Lambda$ replaced by $\Lambda\llbracket t\rrbracket$ as well as 
\begin{enumerate}
	\item[$(1')$] $F\circ \sfp_n=\sfp_n\circ F=F(t^n,x_1^n,x_2^n,\cdots)$ for $n\geq1$,  and $t\circ F=t$. 
\end{enumerate}

For $F\in \Lambda$ or $\Lambda\llbracket t\rrbracket$ with no constant term, the \emph{plethystic exponential} of $F$ is defined to be
\beq\label{y13} \Exp(F)=\sum_{r\geq 0}\sfh_r\circ F=\exp\sum_{r\geq1}\frac{1}{r}\sfp_r\circ F\eeq
where the second equality holds since $\sum_{r\geq0}\sfh_r=\exp\sum_{r\geq1}\frac{1}{r}\sfp_r$.

\subsection{Characteristic polynomial of a symmetric function}\label{S3.2}

In this subsection, we introduce the Stanley map and the characteristic polynomials of symmetric functions. 
We further characterize the Stanley map as the only algebra homomorphism that converts plethysm to composition. 

\begin{definition}\label{def:charpoly}
	Let $F\in \Lambda_n$ be a symmetric function of degree $n$ such that $F=\sum_{\lambda \vdash n}c_\lambda \sfp_\lambda$.  The \emph{characteristic polynomial} of $F$ is defined to be
	\beq\label{y3}\stan_F(m)=\sum_{\lambda \vdash n }c_\lambda m^{\ell(\lambda)} \ \ \in \QQ[m]\eeq
	where $\ell(\lambda)$ denotes the length of a partition $\lambda$. For 
	$F=\sum_{k}F_{k}t^k\in \Lambda_n[t]$ 
	with $F_k\in \Lambda_n$ and  a formal variable $t$, 
	we define
	\[\stan_F(m,t)=\sum_{k}\stan_{F_{k}}(m)t^k\ \ \in \QQ[m,t].\] 
    \end{definition}
    Note that the degree of $\stan_F$ is at most $n$.  It is easy to see that for any positive integer $m_0$,  $\stan_F(m_0)$ is obtained from $F$ by the substitution \eqref{4}. 
By \eqref{y4}, \eqref{y3} can be also written as 
	\beq\label{3}\stan_F(m)=\sum_{\lambda \vdash n}\centz_\lambda^{-1}\langle F,\sfp_\lambda\rangle m^{\ell(\lambda)}.\eeq

By Definition~\ref{def:charpoly}, we have the map 
\beq\label{y9} \stan:\Lambda_n\lra \QQ[m], \quad F\mapsto \stan_F(m)\eeq
and its extension
\beq\label{y10} \widehat \stan: \Lambda\lra \Q[m]\llbracket q \rrbracket, \quad (F\in \Lambda_n)\mapsto \stan_F(m) q^n\eeq
where $q$ is a formal variable recording the degree $n$ of $F$. 
Note that \eqref{y10} is a $\QQ$-algebra homomorphism, called the \emph{principal specialization} in \cite[\S7.8]{Stabook},
or the \emph{Stanley map}.

A remarkable feature of \eqref{y9} and \eqref{y10} is that they convert plethysm to composition. 
\begin{proposition}\label{p:plethysm}
	For any $F\in \Lambda_n$ and $G\in \Lambda_{n'}$ with $n,n'\geq 0$,
	\beq\label{1}\stan_{F\circ G}(m)=(\stan_F\circ \stan_G)(m)=\stan_F(\stan_G(m)).\eeq
\end{proposition}
\begin{proof}
	By definition,  $\stan_{\sfp_r\circ G}=\stan_G$ for any $r\geq 1$ and $G\in \Lambda_{n'}$.
	Moreover, if $F=\sum_{\lambda\vdash n}a_\lambda \sfp_\lambda$, then $\stan_F(m)=\sum_{\lambda\vdash n} a_\lambda m^{\ell(\lambda)}$. Consequently,
\[\stan_{F\circ G}(m)=\sum_{\lambda\vdash n}a_\lambda \prod_{i=1}^{\ell(\lambda)}\stan_{\sfp_{\lambda_i}\circ G}(m)=\sum_{\lambda\vdash n}a_\lambda(\stan_G(m))^{\ell(\lambda)}=\stan_F(\stan_G(m))\]
for any $G\in \Lambda_{n'}$, where $\lambda=(\lambda_1,\cdots, \lambda_{\ell(\lambda)})$.
\end{proof}

\def\hstan{\widehat\stan}
If we extend the composition $(f\circ g)(m)=f(g(m))$ to an operation on $\QQ[m]\llbracket q \rrbracket$ by 
\beq\label{y12} \left(\sum_i f_i(m)q^i\right)\circ \left(\sum_j g_j(m)q^j\right)=\sum_{i,j}f_i(g_j(m))q^{ij},\eeq 
Proposition~\ref{p:plethysm} is equivalent to
\beq\label{y11}  \hstan_{F\circ G}=\hstan_F\circ \hstan_G \in \QQ[m]\llbracket q \rrbracket \quad\text{for }F,G\in \Lambda.\eeq
Conversely,  \eqref{1} or \eqref{y11} determines $\widehat\stan$ and thus $\stan$.
\begin{theorem} 
\label{thm:characterize.chi}
	The map $\widehat \stan :\Lambda\to \Q[m]\llbracket q\rrbracket$ is the unique $\Q$-algebra homomorphism that sends $\Lambda_n$ to $\QQ[m]q^n$ for every $n$ and satisfies \eqref{y11}.
\end{theorem}
	\begin{proof}
Let $\hstan'$ be another such homomorphism and let $\stan':\Lambda_n\to \QQ[m]$ be defined by $\hstan'_F=\stan'_F(m)q^n$ for $F\in \Lambda_n$.  
Since $\sfp_n\circ a=a$ and $\stan'_a=a$  for any $n\geq1$ and $a\in \Lambda_0=\Q$, 
	\[\widehat\stan'_{\sfp_n}(a)=\widehat\stan'_{\sfp_n}(\hstan'_a(m))=\widehat\stan'_{\sfp_n\circ a}(m)=\stan'_a(m)q^n=aq^n.\]
Therefore, $\stan'_{\sfp_n}(m)=m$ and $\hstan'_{\sfp_n}=mq^n$ for $n\geq 1$. The uniqueness follows from this because $\hstan'$ is a $\QQ$-algebra homomorphism. 
	\end{proof}

All these extend to a graded version with a new variable $t$ as 
\beq\label{7}\widehat \stan:\Lambda\llbracket t\rrbracket \lra \Q[m]\llbracket q,t\rrbracket, \quad \sfp_nt^k\mapsto mq^nt^k.\eeq
By abuse of notation, we denote this map also by $\widehat \stan$. 
If we extend the composition to an operation $\QQ[m]\llbracket q \rrbracket\times \QQ[m]\llbracket q,t \rrbracket\to \QQ[m]\llbracket q,t \rrbracket$ by 
\beq\label{y12a} \left(\sum_i f_i(m)q^i\right)\circ \left(\sum_{j,k} g_{jk}(m)q^jt^k\right)=\sum_{i,j,k}f_i(g_{jk}(m))q^{ij}t^{ik},\eeq 
Proposition~\ref{p:plethysm} upgrades to 
\beq\label{y11a}  \hstan_{F\circ G}=\hstan_F\circ \hstan_G \in \QQ[m]\llbracket q,t \rrbracket \quad\text{for }F\in \Lambda,~G\in \Lambda\llbracket t\rrbracket.\eeq
\begin{theorem}\label{thm:characterize.hatchi}
	The map \eqref{7} is the unique $\Q$-algebra homomorphism 
	that sends $\Lambda_n\llbracket t\rrbracket$ to $\Q[m]\llbracket t \rrbracket q^n$ for every $n$ and satisfied \eqref{y11a}.
\end{theorem}
\begin{proof}
	The proof is basically the same as that for   Theorem~\ref{thm:characterize.chi}.
\end{proof}

Finally we observe that $\stan$ satisfies a duality with respect to the involution \eqref{y5} on $\Lambda$ (cf.~\cite[Exercise~24]{StabookEx}). 
\begin{lemma}[$\omega$-duality]\label{l:omega}
	For $F\in \Lambda_n$, $\stan_{\w F}(m)=(-1)^n\stan_F(-m)$.
\end{lemma}
\begin{proof}
	This follows immediately from $\omega\sfp_\lambda=(-1)^{n-\ell(\lambda)}\sfp_\lambda$ for $\lambda\vdash n$.
\end{proof}

\subsection{Plethysm on polynomials}
For convenience, we introduce the plethysm operation on polynomials as follows. 
\begin{definition} 
For any $F\in \Lambda_n$, we define the plethysm 
\[F\circ (-):\Q[m]\llbracket q,t\rrbracket\lra \Q[m]\llbracket q,t\rrbracket,\quad  
	F\circ \left( f(m)q^it^j\right)=\stan_F(f(m)) q^{ni}t^{nj}.\]
For $F\in \Lambda$, we extend it linearly. 
\end{definition}

It is straightforward to check that $F\circ (-)$ is uniquely characterized by the following properties: for $f\in \Q[m]\llbracket q,t\rrbracket$ and $F,G\in \Lambda$,
\begin{enumerate}
	\item $\sfp_n\circ f=f^{[n]}$ for $n\geq 1$;
	\item $(F+G)\circ f=F\circ f+G\circ f$;
	\item $(FG)\circ f=(F\circ f)(G\circ f)$
\end{enumerate}
where 
$f^{[n]}(m,q,t)=f(m,q^n,t^n)$. 
Note that $\sfp_n\circ(-)$  leaves $m$ unchanged, while $q$ and $t$ are raised to the $n$-th power. If $F=\sum_{\lambda \vdash n }c_\lambda\sfp_\lambda$, then
\[F\circ f =\sum_{\lambda \vdash n}c_\lambda f^{[\lambda]}\]
where we write $f^{[\lambda]}:=f^{[\lambda_1]}\cdots f^{[\lambda_\ell]}$ for $\lambda=(\lambda_1,\cdots, \lambda_\ell)$.

For $f\in \Q[m]\llbracket q,t\rrbracket$ with no constant term, as in \eqref{y13}, we define the \emph{plethystic exponential}  of $f$  to be  
\beq\label{y14} 
\Exp(f)=\sum_{r\geq0}\sfh_r\circ f=\exp\sum_{r\geq 1}\frac{1}{r}\sfp_r\circ f=\exp\sum_{r\geq 1}\frac{1}{r}f^{[r]}.\eeq
\begin{example}[Monomial]\label{ex:Exp.monomial}
    For integers $a,j,n,k$ with $j,n,k\geq0$, we have
    \[
    \begin{split}
        \Exp(am^jq^nt^k)&=\exp \sum_{r\geq1}\frac{1}{r}am^jq^{rn}t^{rk}\\
        &=\exp(- am^j \log (1-q^nt^k))=(1-q^nt^k)^{-am^j}
    \end{split}
    \]
    provided $(n,k)\neq (0,0)$.
\end{example}
\begin{remark} [Alternative description of $\Exp (f)$]
\label{ex:Exp}
    Recall that  
    \beq\label{eq:h.p.expansion}\sfh_n=\sum_{\lambda\vdash n}\centz_\lambda^{-1}\sfp_\lambda\eeq
    for any $n\geq 1$. 
    Then $\Exp(f)$ can also be written as
    \[\Exp(f)=\sum_{r\geq0}\sfh_r\circ f
    =\sum_{\lambda \in \Par}\centz_\lambda^{-1}f^{[\lambda]},\]
    where $\Par$ denotes the set of partitions of nonnegative integers and $f^{[0]}=1$.
\end{remark}

\subsection{Characteristic polynomial of an $\symS_n$-representation}\label{S3.3}
In this subsection, we define the characteristic polynomial of an $\symS_n$-module $V$ by applying the Stanley map $\stan$ to the Frobenius characteristic of $V$. 

\medskip

Let $R_n=\mathrm{Rep}(\symS_n)_\QQ$ be the $\QQ$-vector space spanned by irreducible $\symS_n$-representations and let $R_0=\Q$.
Then $R=\prod_{n\geq0} R_n$ is the ring of virtual representations of the symmetric groups whose ring structure is given by \[V.W=\mathrm{Ind}_{\symS_n\times \symS_m}^{\symS_{n+m}} V\otimes W\] for an $\symS_n$-module $V$ and an $\symS_m$-module $W$. 
Then it is well-known that the Frobenius characteristic map  
\beq\label{eq:Frob}
\ch:R\mapright{\cong}\Lambda, \quad \ch(V)=\ch_{\symS_n}(V)=\frac{1}{n!}\sum_{\sigma\in \symS_n}\mathrm{Tr}_V(\sigma)\sfp_{\sigma}\ \ \text{ for }V\in R_n\eeq
is a ring isomorphism  where 
$\sfp_\sigma:=\sfp_{\lambda(\sigma)}$ and $\lambda(\sigma)\vdash n$ is 
 the cycle partition of $\sigma$. This sends the Specht module $S^\lambda$ to the Schur function $\sfs_\lambda$.

\smallskip

For $\sigma \in \symS_n$, let $\ell(\sigma)$ denote the length of the partition $\lambda(\sigma)$.
\begin{definition}\label{y8}
	The \emph{characteristic polynomial} of an $\symS_n$-representation $V$ 
	is 
	\[\stan_V(m)=\stan_{\ch(V)}(m)=\frac{1}{n!}\sum_{\sigma \in \symS_n}\mathrm{Tr}_V(\sigma)m^{\ell(\sigma)}.\]
	When $V=\bigoplus_{k\in \Z}V_k$ is a graded $\symS_n$-module, 
	we let $\stan_V(m,t)=\sum_{k}\stan_{V_k}(m)t^k$.
\end{definition}

The coefficients of $\stan_V$ keep essential information about $V$. 
\begin{lemma} \label{l:mult}
	Let $V\in R_n$ be an $\symS_n$-representation. 
	\begin{enumerate}
		\item The coefficient of $m^n$ in $\stan_V(m)$ is $\frac{1}{n!}\dim\, V$.
		\item The value $\stan_V(1)$ of $\stan_V$ at $m=1$ is the dimension of the invariant subspace $V^{\symS_n}$.
		\item $(-1)^n\cdot \stan_V(-1)$ is the multiplicity of the sign representation $\sgn$ in the decomposition of $V$ into irreducible representations.
	\end{enumerate}
\end{lemma}
\begin{proof}
	(1) is a direct consequence of the facts that $\ell(\sigma)=n$ if and only if $\sigma$ is the identity $e$ and $\mathrm{Tr}_V(e)=\dim\, V$.
	
	(2) follows from \eqref{3} and the identity $\ch_{\symS_n}(\mathbbm{1})=\sfh_n=\sum_{\lambda \vdash n}\centz_\lambda^{-1}\sfp_\lambda$ where $\mathbbm{1}$ denotes the trivial representation. Indeed,  $$\stan_V(1)=\sum_{\lambda\vdash n}\centz_{\lambda}^{-1}\langle \ch_{\symS_n}(V),\sfp_\lambda\rangle=\langle\ch_{\symS_n}(V),\ch_{\symS_n}(\mathbbm{1})\rangle=\langle V,\mathbbm{1}\rangle=\dim\, V^{\symS_n}.$$ 
	
	(3) follows from (2) and Lemma~\ref{l:omega}, since $\w \ch_{\symS_n}(V)=\ch_{\symS_n}(V\otimes \sgn)$.
\end{proof}

\subsection{First examples} 
In this subsection, we work out examples of the characteristic polynomials of $\symS_n$-modules. 

\begin{example}\cite[Example~I.2.1]{Mac}
\label{ex:trivial}
When $V=\mathbbm{1}$ is the trivial representation, $\ch_{\symS_n}(\mathbbm{1})=\sfh_n$. By \eqref{4}, we have 
\[\stan_{\sfh_n}(m)=\multiset{m}{n} 
:=\binom{n+m-1}{n}=\frac{m(m+1)\cdots(m+n-1)}{n!}\]
since $\sfh_n=\sum_{1\leq i_1\leq \cdots \leq i_n}x_{i_1}\cdots x_{i_n}$.
Similarly, when $V=\sgn$ is the sign representation, $\ch_{\symS_n}(\sgn)=\sfe_n$ and
\[\stan_{\sfe_n}=\binom{m}{n}=\frac{m(m-1)\cdots(m-n+1)}{n!}\]
since $\sfe_n=\sum_{1\leq i_1 < \cdots  < i_n}x_{i_1}\cdots x_{i_n}$.
\end{example}

\begin{example}[Permutation modules] \label{ex:permutation.module} 
The \emph{permutation $\symS_n$-module}  
associated to a partition $\lambda=(\lambda_1,\cdots, \lambda_\ell)$ of $n$ is the induced representation 
\[M^\lambda:=\Ind^{\symS_n}_{\symS_\lambda}\mathbbm{1}\] of the trivial representation of the Young subgroup $\symS_\lambda=\symS_{\lambda_1}\times\cdots\times\symS_{\lambda_\ell}$. 
Its dimension is ${n!}/{\lambda!}$ where $\lambda! =\lambda_1!\cdots \lambda_\ell !$ and we have 
$$\ch_{\symS_n}(M^\lambda)=\sfh_\lambda=\sfh_{\lambda_1}\cdots \sfh_{\lambda_\ell}.$$
Hence by multiplicativity,
\beq\label{eq:permutation.module}\begin{split}
	\stan_{M^\lambda}(m)&=\prod_{i=1}^\ell \multisetBig{m}{\lambda_i} 
	=\frac{1}{\lambda!}m^{\lambda^t_1}(m+1)^{\lambda^t_2}\cdots (m+{\lambda_1}-1)^{\lambda^t_{\lambda_1}}
\end{split}\eeq
where $\lambda^t$
denotes the dual partition of $\lambda$ with parts $\lambda^t_1\geq \cdots \geq \lambda^t_{\lambda_1}>0$. Note that $\lambda^t_1=\ell(\lambda)$ is the length of $\lambda$ and \eqref{eq:permutation.module} is divisible by $m^{\ell(\lambda)}$.
\end{example}

From this, we deduce the following.
\begin{lemma} \label{l:M.mult}
	Let $V$ be an $\symS_n$-representation with 
	$V=\sum_{\lambda\vdash n}b_\lambda M^\lambda$. Then 
	\begin{enumerate}
		\item $b_{(n)}/n$ is equal to the coefficient of $m$ in $\stan_V(m)$ and
		\item $b_{(1^n)}$ is equal to $(-1)^n\cdot \stan_V(-1)=\dim~(V\otimes \sgn)^{\symS_n}$.
	\end{enumerate}
\end{lemma}
\begin{proof} By \eqref{eq:permutation.module}, we find that, for $\lambda=(\lambda_1,\cdots,\lambda_\ell)$,
	\begin{enumerate}
		\item[$(1')$] $\stan_{M^\lambda}(m)$ is divisible by $m^{\ell}$ and the coefficient of $m^\ell$ is ${1}/{(\lambda_1\cdots \lambda_n)}$;
		\item[$(2')$] $\stan_{M^\lambda}(-1)\neq 0$ if and only if $\lambda=(1^n)$ and $\stan_{M^{(1^n)}}(-1)=(-1)^n$.
	\end{enumerate}	
	The statements (1) and (2) follow immediately from $(1')$ and $(2')$.
\end{proof}

\begin{example}[Permutation representations]\label{ex:chi.U_H}
	A \emph{permutation representation} of $\symS_n$ is a representation of $\symS_n$ that admits an $\symS_n$-invariant basis. It is called \emph{transitive} if an (or equivalently any) $\symS_n$-invariant basis consists of a single orbit. Every permutation representation of $\symS_n$ decomposes as a direct sum of transitive ones and every transitive permutation representation is the induced representation  
	\beq\label{y16} U_H:=\Ind^{\symS_n}_H\mathbbm{1}\eeq 
	of the trivial representation of some subgroup $H$ of $\symS_n$. Moreover, any $\symS_n$-invariant basis of $U_H$ is isomorphic to $\symS_n/H$ as an $\symS_n$-set. In particular,
	\[\begin{split}
		\ch_{\symS_n}(U_H)&=\frac{1}{n!}\sum_{\sigma \in \symS_n}\left\lvert(\symS_n/H)^\sigma\right\rvert \cdot \sfp_{\sigma}\\
		&=\frac{1}{n!}\sum_{\sigma \in \symS_n}\left\lvert\{\tau H\in \symS_n/H:\sigma \in \tau H\tau^{-1}\}\right\rvert\cdot  \sfp_{\sigma}\\
		&=\frac{1}{n!}\sum_{\tau H \in \symS_n/H}\sum_{\sigma \in \tau H \tau^{-1}}\sfp_{\sigma}\\
		&=\frac{1}{n!}\sum_{\tau H \in \symS_n/H}\sum_{\sigma' \in  H }\sfp_{\sigma'}=\frac{1}{\lvert H \rvert}\sum_{\sigma' \in H}\sfp_{\sigma'}
	\end{split}
	\]
	where the first equality holds since the $\sigma$ acts on $U_H$ as a permutation on the basis $\symS_n/H$ and the second equality holds since $\sigma \tau H=\tau H$ if and only if $\sigma \in \tau H\tau^{-1}$. The fourth equality follows from setting $\sigma=\tau \sigma'\tau^{-1}$, since $\sfp_\sigma=\sfp_{\sigma'}$, while the third and the last equalities are immediate. 

	Consequently, the characteristic polynomial of $U_H$ is 
	\beq\label{eq:chi.U_H}\stan_{U_H}(m)=\frac{1}{\lvert H\rvert}\sum_{\sigma\in H}m^{\ell(\sigma)}.\eeq
	In other words, the coefficient of $m^\ell$ in $\stan_{U_H}(m)$ is the number of permutations contained in $H$ that have precisely $\ell$ orbits in $[n]:=\{1,\cdots, n\}$, divided by $\lvert H\rvert$. 
	
	We note that for a positive integer $m_0$,  the value $\stan_{U_H}(m_0)$ is equal to the number of ways to color the elements in $[n]$ using $m_0$ colors, where two colorings are considered equivalent if they belong to the same $H$-orbit under the natural $H$-action on $[n]$ (see \cite[Corollary~7.24.6]{Stabook}).
\end{example}
\begin{remark}
	For a permutation representation $V$ 
	of $\symS_n$, $\stan_V$ always has nonnegative coefficients. 
    On the other hand, 
    it may have internal zeros in its coefficients. For example, for $H=\langle (1,2,3)\rangle \subset \symS_3$,  $\stan_{U_H}=\frac{1}{3}(m^3 +m)$.   
\end{remark}

\begin{example}[Specht module]\label{ex:chi.schur} 
Let $\lambda=(\lambda_1,\cdots,\lambda_\ell)$ be a partition of $n$ with $\lambda_1\geq \cdots \geq \lambda_\ell>0$. Then for $m\geq \ell$,
	\[\stan_{\sfs_\lambda}(m)=\prod_{1\leq i<j\leq m}\left(1+\frac{\lambda_i-\lambda_j}{j-i}\right).\]
	In particular, $\stan_{\sfs_\lambda}(n)$ is equal to the dimension of the irreducible $GL_n$-representation $V_\lambda$ with the highest weight $\lambda$, by the Weyl dimension formula. See \cite[(7.106) or (A.2.4)]{Stabook}.
\end{example}

\begin{example}
\label{ex:chi.length.two}
	Suppose that an $\symS_n$-representation $V$ is spanned by $M^\lambda$ with $\ell(\lambda)\leq 2$. 
	In this case, the representation $V$ is completely determined by its characteristic polynomial $\stan_V(m)$. To see this, write 
\[V=\sum_{0\leq j \leq \frac{n}{2}}c_jM^{(n-j,j)}, \quad c_j\in \ZZ\] 
so that 
$$\stan_V(m)=\sum_{1\leq j\leq \frac{n}{2}}c_j \stan_{\sfh_{(n-j,j)}}(m), \quad\text{where }~ 
\stan_{\sfh_{(n-j,j)}}=\multisetBig{m}{n-j}\multisetBig{m}{j}.$$
Letting $r=\lfloor \frac{n}{2}\rfloor$, we have 
\[B_{ij}:=\stan_{\sfh_{(n-j,j)}}(i-n)=\begin{cases}
	0 &\text{ if }i>j\\
	(-1)^n\binom{n-i}{n-j}\binom{n-i}{j} & \text{ if }i\leq j
\end{cases}\]
for $0\leq i\leq r$,
since $\multisetbody{i-n}{n-j}=(-1)^{n-j}\binom{n-i}{n-j}$ and $\multisetbody{i-n}{j}=(-1)^j\binom{n-i}{j}$.
The upper triangular matrix $B=(B_{ij})_{0\leq i,j\leq r}$ has nonzero determinant $(-1)^{n(r+1)}\prod_{i=0}^r \binom{n-i}{i}\neq 0$. Since $\stan_V(i-n)=\sum_{j=0}^r B_{ij}c_j$ for every $0\leq i\leq r$, $c_j$ are determined by the matrix multiplication
\[(c_0,\cdots, c_r)^t=B^{-1}\cdot (\stan_V(-n),\cdots, \stan_V(r-n))^t. \]
\end{example}

\bigskip

\section{Characteristic polynomials of varieties and log-concavity}\label{s:examples}
In this section, we present examples derived from geometry. Let $Z$ be a smooth projective pure-dimensional variety with no odd degree cohomology, i.e. $H^{2k+1}(Z)=0$ for every $k$. Suppose that $Z$ is equipped with an $\symS_n$-action, or more generally, $H^{2k}(Z)$ is equipped with an $\symS_n$-action for every $k$. We denote the characteristic polynomial of the induced graded $\symS_n$-module $H^*(Z)$ by
\[\stan_Z(m,t)=
\sum_{k\geq0}\stan_{\ch_{\symS_n}H^{2k}(Z)}(m)t^k.\]
We call this \emph{the characteristic polynomial of $Z$}.

Note that $\stan_Z(m,1)$ is the characteristic polynomial of the ungraded $\symS_n$-representation $H(Z)=\bigoplus_kH^{2k}(Z)\in R_n$, while $\stan_{Z}(1,t)$ is the Poincar\'e polynomial of the invariant part $H^*(Z)^{\symS_n}$,
which is isomorphic to $H^*(Z/\symS_n)$ when there is an $\symS_n$-action on $Z$.

\subsection{$n$-fold products} \label{S4.1}
Let $X$ be a smooth projective variety with no odd cohomology.
Consider the $\symS_n$-action on the $n$-fold product $X^n=X\times \cdots\times X$ of $X$ given by permuting the product components. This can be viewed as the space of
ordered $n$ points in $X$ acted on by an $\symS_n$-action reordering them.

We compute $\stan_{X^n}(m,t)$ and explore the four types of log-concavities introduced in Definition~\ref{def:lc}. 
We begin with an example where $X=\PP^1$. 

\begin{example}\label{ex:nfold.P1}
	Let $X=\PP^1$. Observe that \[\sum_{k\geq0}\ch_{\symS_n}H^{2k}(X^n)t^k=\sfh_n\circ (\sfh_1+\sfh_1t)=\sum_{k=0}^n \sfh_{n-k}\sfh_k t^k .\]
	Applying the Stanley map $\stan$, we obtain
	\beq\label{eq:P1a}\stan_{(\PP^1)^n}(m,t)=\sum_{k=0}^n\multisetBig{m}{n-k}\multisetBig{m}{k}t^k.\eeq
	In particular, when evaluated at $t=1$ or $m=1$, one gets
	\beq\label{eq:P1b}\begin{split}
		&\stan_{(\PP^1)^n}(m,1)=\sum_{k=0}^n\multisetBig{m}{n-k}\multisetBig{m}{k}=\multisetBig{2m}{n};\\
		&\stan_{(\PP^1)^n}(1,t)=\sum_{k=0}^nt^k
	\end{split}
	\eeq
	respectively. By \eqref{eq:P1a} and \eqref{eq:P1b}, we find that 
	\begin{itemize}
		\item $\stan_{(\PP^1)^n}(m,t)$ is length log-concave in degree $k$ for any $0\leq k\leq n$;
		\item $\stan_{(\PP^1)^n}(m,t)$ is length log-concave at $t=1$;
		\item $\stan_{(\PP^1)^n}(m,t)$ is degree log-concave at $m=1$,
	\end{itemize}	
	respectively, where the first and the second statements follow by Lemma~\ref{lem:lc.product} or \ref{lem:realrooted}.
	However, $\stan_{(\PP^1)^n}(m,t)$ is not degree log-concave in a given length in general. For example, when $n=4$, 
   \begin{align*}
       \stan_{(\PP^1)^4}(m,t)&=\frac14(t^4+1)m + \frac{1}{24}(11t^4+8t^3+6t^2+8t+11)m^2 \\ &+ \frac{1}{4}(1+2t+2t^2+2t^3+t^4)m^3 + \frac{1}{24}(t^4+4t^3+6t^2+4t+1)m^4,
   \end{align*}
   which is not degree log-concave in length 2. 
   
	On the other hand, interestingly, it seems to hold that 
	\begin{itemize}
		\item $\stan_{(\PP^1)^n}(m,t)$ is degree log-concave in length $\ell$ whenever the inequalities $\lceil \log_2(n\log_2n)\rceil \leq \ell\leq n$ are satisfied. 
	\end{itemize}
	We have checked this for $n\leq 200$ by numerical computation. We also note that the bound $\lceil \log_2(n\log_2n)\rceil$ is not sharp.
\end{example}

For general $X$, we can explicitly calculate $\stan_{X^n}(m,t)\in \Q[m,t]$, 
and study the log-concavity.
Let $p_X=\sum_{k=0}^d b_kt^k$ denote the Poincar\'e polynomial of $X$, where $d=\dim \,X$ and $b_k=\dim\, H^{2k}(X)$.

\begin{theorem}\label{thm:prod.gr.gen}
	We set $X^0$ to be a point. Then, 
	\beq \label{eq:genprod}\sum_{n\ge 0}\sum_{k\geq0}\ch_{\symS_n}H^{2k}(X^n)t^k=\Exp(p_X\sfh_1).\eeq
	Consequently, the generating series of $\stan_{X^n}(m,t)$ is
	\beq\label{eq:gen.series.nfold} 
    \sum_{n\geq 0}\stan_{X^n}(m,t)q^n
	=\prod_{i=0}^d\left(1-qt^i\right)^{-b_im}.\eeq
\end{theorem}
\begin{proof}
	By the K\"unneth formula and the property of the plethysm by $\sfh_n$, we have
	\beq\label{eq:hncirc}
		\sum_{k\ge 0} \ch_{\symS_n}H^{2k}(X^n)t^k=\sfh_n\circ \left(\sum_{k \ge 0} \ch_{\symS_1}H^{2k}(X)t^k \right)=\sfh_n\circ (p_X\sfh_1).
	\eeq
    Hence, we obtain \eqref{eq:genprod} by the definition of plethystic exponential.

The formula \eqref{eq:gen.series.nfold} follows from \eqref{eq:genprod}, by applying $\widehat\stan_{(-)}$ since 
	\[\begin{split}
		\widehat \stan_{\Exp(p_X \sfh_1)}&=\exp\sum_{r\geq1}\frac{1}{r}\sfp_r\circ \widehat\stan(p_X\sfh_1)=\exp m\sum_{r\geq 1}\frac{1}{r}\sfp_r\circ(p_Xq)=\left(\Exp(p_Xq)\right)^m
	\end{split}
	\]
	which is equal to $\prod_{i=0}^d(1-qt^i)^{-b_im}$ by Example~\ref{ex:Exp.monomial}.
\end{proof}

\begin{corollary} \label{cor:nfold.formula} Let $X$ be as above.
  \begin{enumerate}
    \item	As an element of $\Q[m][t]$, $\stan_{X^n}(m,t)$ is written as
	\beq \label{eq:nfold.coeff}\stan_{X^n}(m,t)=\sum_{n_0+\cdots+n_d=n}\multiset{b_0m}{n_0}\cdots \multiset{b_dm}{n_d}t^{ 1\cdot n_1+\cdots +d\cdot n_d}.\eeq
    
    \item When evaluated at $t=1$, we have
	\beq\label{eq:nfold.t1}\stan_{X^n}(m,1)=\multisetBig{Nm}{n}=\frac{1}{n!}Nm(Nm+1)\cdots (Nm+n-1),\eeq
    where $N=\sum_{k=0}^{d}b_k$.
    
    \item When evaluated at $m=1$, we have  
	\beq\label{eq:nfold.inv}\stan_{X^n}(1,t)=\sum_{n_0+\cdots+n_d=n}\multiset{b_0}{n_0}\cdots \multiset{b_d}{n_d}t^{1\cdot n_1+\cdots +d\cdot n_d}.\eeq
    
    \item As an element of $\Q[t][m]$, $\stan_{X^n}(m,t)$ is written as
	\beq\label{eq:nfold.coeffm}
    \stan_{X^n}(m,t)=\sum_{\lambda\vdash n}\centz_\lambda^{-1}p_X^{[\lambda]}m^{\ell(\lambda)}\eeq
    where $p_X^{[\lambda]}
    =\prod_{i=1}^{\ell(\lambda)}p_X(t^{\lambda_i})$ for a partition $\lambda=(\lambda_1,\ldots,\lambda_{\ell(\lambda)})$ of $n$.
    \end{enumerate}
\end{corollary}

\begin{proof}
    (1), (2) and (3) are immediate from \eqref{eq:gen.series.nfold} by reading the coefficient of $q^n$.
    For (4), observe that from \eqref{eq:hncirc},
	\[\sum_{k\ge 0} \ch_{\symS_n}H^{2k}(X^n)t^k=\sfh_n\circ (p_X\sfh_1)=\sum_{\lambda\vdash n}\centz_\lambda^{-1}\sfp_\lambda \circ(p_X\sfh_1)=\sum_{\lambda\vdash n}\centz_\lambda^{-1}p_X^{[\lambda]}\sfp_\lambda.\]
	Applying $\stan$ to this, \eqref{eq:nfold.coeffm} follows.
\end{proof}

By Corollary~\ref{cor:nfold.formula}, we obtain the following straightforward log-concavity results. 
\begin{corollary}
    \begin{enumerate}
        \item $\stan_{X^n}(m,t)$ is length log-concave at $t=1$. 
        \item $\stan_{X^n}(m,t)$ is degree log-concave in length $n$, provided that $p_X$ is log-concave with no internal zeros. 
        \item The coefficient of $m$ in $\stan_{X^n}(m,t)$ has internal zeros whenever $n>1$.
    \end{enumerate}
\end{corollary}
\begin{proof}
    (1) follows from \eqref{eq:nfold.t1} and Lemma~\ref{lem:realrooted}. For (2), by \eqref{eq:nfold.coeffm}, the coefficient of $m^n$ in $\stan_{X^n}(m,t)$ is $\frac{1}{n!}(p_X(t))^n$. This is log-concave with no internal zeroes when $p_X$ is. 
    For (3), by \eqref{eq:nfold.coeffm}, the coefficient of $m$ in $\stan_{X^n}$ is $\frac{1}{n}p_X(t^n)=\frac{1}{n}\sum_{i=0}^d b_it^{ni}$.
\end{proof}

Note that there are examples of $X$ with non-log-concave Poincar\'e polynomial. For example, when $X=\mathrm{Gr}(2,4)$, $p_X(t)=1+t+2t^2+t^3+t^4$ is not log-concave.

\begin{example}\label{ex:nfold.proj.sp}
	When $X=\PP^d$, \eqref{eq:nfold.inv} reads
	\beq \label{eq:nfold.proj.sp}\begin{split}
		\stan_{X^n}(1,t)&=\sum_{\substack{n_0+\cdots+n_d=n}}t^{1\cdot n_1+\cdots +d\cdot n_d}=\sum_{\substack{\lambda\in \mathrm{Par},\\ \ell(\lambda)\leq n,\lambda_1\leq d}}t^{\lvert \lambda\rvert}={n+d \brack n}_t
	\end{split}\eeq
	where $\lambda$ runs over all the partitions with length at most $n$ and all parts at most $d$ and 
	for $a\geq b\geq0$, 
	 \[{a \brack b}_t:=\frac{[a]_t!}{[a-b]_t![b]_t!}~\in \Z[t]\]
	 with $[a]_t!=\prod_{i=1}^a[i]_t$ and $[a]_t=\sum_{i=0}^{a-1}t^i$ for all $a\geq0$. 
      The second equality in \eqref{eq:nfold.proj.sp} holds by identifying the $(d+1)$-tuple $(n_0,\cdots, n_d)$ with the partition $\lambda=1^{n_1}\cdots d^{n_d}$ and the last equality 
      holds by \cite[Proposition~1.7.3]{Stabook1}.
	 
	 Unlike the case of $X=\PP^1$, this is not log-concave in general. For example, when $n=d=2$, one gets $\stan_{(\PP^2)^2}(1,t)={4\brack 2}_t=1+t+2t^2+t^3+t^4$. In other words, $\stan_{X^n}(m,t)$ is not degree log-concave at $m=1$ in general. 
\end{example}

Recall that  when $X=\PP^1$, we verified in Example~\ref{ex:nfold.P1} that $\stan_{X^n}(m,t)$ 
is degree log-concave in length $\lceil \log_2(n\log_2n)\rceil\leq \ell\leq n$ for $n\leq 200$, while it fails to be degree log-concave in smaller lengths. Our goal is to investigate the degree log-concavity of $\stan_{X^n}(m,t)$ in sufficiently large length. Additionally, we aim to examine the length log-concavity of $\stan_{X^n}(m,t)$ at a given degree.

\begin{question} Let $\stan_{X^n,k}^\ell\in \Q$ denote the coefficient of $m^\ell t^k$ in $\stan_{X^n}(m,t)$ for each $\ell,k$. 

	(1) Is it true that 
	\[(\stan_{X^n,k}^\ell)^2\geq \stan_{X^n,k}^{\ell+1}\cdot\stan_{X^n,k}^{\ell-1} \quad \text{ for any $\ell,k$}?\]
	
	(2) Assume that $p_X$ is log-concave with no internal zeros. Is there a lower bound for $\ell$ as a function of $n$ such that 
	\[(\stan_{X^n,k}^\ell)^2\geq \stan_{X^n,k+1}^{\ell}\cdot\stan_{X^n,k-1}^{\ell} \] 
    holds for all $k$?
\end{question}
Here we give a partial asymptotic answer to these questions. Indeed, both inequalities hold when $k$ is sufficiently small and $\ell$ is sufficiently large with respect to $n$.
\begin{proposition}
	Let $j,k\geq 0$. The coefficient of $m^{n-j}t^k$ in $\stan_{X^n}(m,t)$ is
	\[
    \stan_{X^n,k}^{n-j}
    =\frac{b_0^{n-k-j}b_1^k}{2^jj!k!}\cdot\frac{n^{2j+k}}{n!}+o\Big(\frac{n^{2j+k}}{n!}\Big).\]
	In particular, these coefficients satisfy the log-concavities
	\[\begin{split}
		&
        \left(\stan_{X^n,k}^{n-j}\right)^2
        \geq \stan_{X^n,k-1}^{n-j}\cdot\stan_{X^n,k+1}^{n-j} \and
        \left(\stan_{X^n,k}^{n-j}\right)^2\geq \stan_{X^n,k}^{n-j-1}\cdot \stan_{X^n,k}^{n-j+1}
	\end{split}\]
	for sufficiently large $n$. 
\end{proposition}
\begin{proof}
	Recall that $\multisetbody{m}{n}=\sum_{j=0}^{n-1}\frac{c(n,n-j)}{n!}m^{n-j}$, where $c(n,j)$ is the signless Stirling number of the first kind. It follows from \eqref{eq:nfold.coeff} that
\beq\label{eq:nfold.coeff2}\begin{split}
	\stan_{X^n,k}^{n-j}
	&=\sum_{\substack{j_0+\cdots+j_d=j,\\ n_0+\cdots+n_d=n,\\ 1\cdot n_1+\cdots +d\cdot n_d=k}}\prod_{i=0}^d \frac{c(n_i,n_i-j_i)}{n_i!}\cdot b_i^{n_i-j_i}.
\end{split}
\eeq
Since $j_0,\cdots, j_d$ and $n_1,\cdots, n_d$ are bounded by $j$ and $k$ respectively, and it is known that $c(n,n-j)=\frac{1}{2^jj!}n^{2j}+o(n^{2j})$,  
for any given $j\geq 0$, one can easily check that the summands in \eqref{eq:nfold.coeff2} are dominated by those with $n_0$ minimal and $j_0$ maximal, namely, those with 
\[(n_0,\cdots, n_d)=(n-k, k,0,\cdots, 0) \and (j_0,\cdots, j_d)=(j, 0,\cdots, 0). \]   
Indeed, $\frac{c(n_0,n_0-j_0)}{n_0!}=\frac{1}{2^{j_0}j_0!}\cdot\frac{n_0^{2j_0}}{n_0!}+o(\frac{n_0^{2j_0}}{n_0!})$ is maximal if $2j_0-n_0$ is maximal. 
In particular, the coefficient \eqref{eq:nfold.coeff2} is equal to 
\[\begin{split}
	\frac{c(n-k,n-k-j)b_0^{n-k-j}}{(n-k)!}\cdot \frac{c(k,k)b_1^{k}}{k!}&=\frac{b_0^{n-k-j}b_1^k}{2^jj!k!}\cdot\frac{(n-k)^{2j}}{(n-k)!} 
\end{split}
 \]
 modulo $o(\frac{n^{2j+k}}{n!})$. The first formula in the proposition now follows from 
 $\frac{(n-k)^{2j}}{(n-k)!} =\frac{n^{2j+k}}{n!}+o(\frac{n^{2j+k}}{n!}) $. The inequalities are immediate from the asymptotic formula.
\end{proof}

\begin{remark}[Zeta function]
	When $X$ is defined over $\ZZ$, and  has cohomology of pure Tate type, so that $\lvert X(\FF_q)\rvert = p_X(q)$ for any prime power $q$ (for example, this holds if $X$ admits a cell decomposition), the generating series 
    \eqref{eq:gen.series.nfold} is equal to the $m$-th power of the Zeta function
	\[\zeta_X(t)=\exp \sum_{n\geq1}\frac{1}{n}\lvert X(\FF_{q^n}) \rvert t^{n} =\exp \sum_{n\geq1}\frac{1}{n} \sfp_n\circ (p_X(q) t)=\Exp(p_X(q)t)\]
	of $X$ up to interchanging the roles of $q$ and $t$. 
	Here $q$ is regarded as a variable when applying the plethysm. 
	Moreover, the equality 
    \eqref{eq:gen.series.nfold} corresponds to the identity appearing in the Weil conjecture
	\[\zeta_X(t)=\frac{P_1(t)\cdots P_{2d-1}(t)}{P_0(t)\cdots P_{2d}(t)}\]
	with polynomials $P_{2i+1}(t)=1$ and $P_{2i}(t)=(1-q^it)^{b_i}$ for all $i$. 
\end{remark}

\medskip
\subsection{GIT moduli space of $n$ points on $\PP^1$}\label{S4.2} 
Let $n=2r+1$ be odd with $r\geq 1$. We let 
\[Y_n=(\PP^1)^n/\!\!/\SL_2(\C)\]
be the GIT quotient (cf.~\cite{MFK}) of $(\PP^1)^n$ by the diagonal $\SL_2(\C)$ action with respect to the linearization $\sO_{(\PP^1)^n}(1,\cdots, 1)$. 
It is a smooth projective variety whose cohomology was computed in \cite[Example~5.18]{Kir}.
 
The natural action of $\symS_n$ on $(\PP^1)^n$ permuting the factors commutes with that of $\SL_2(\CC)$ and hence we have the induced action of $\symS_n$ on $Y_n$. 
It is straightforward to apply Kirwan's method to compute the graded $\symS_n$-module $H^*(Y_n)$ to obtain (cf.~\cite[Lemma~4.7]{BM})
\beq \label{eq:GIT}
	\sum_{k=0}^{2r-2}\ch_{\symS_n}H^{2k}(Y_n)t^k 
	=\sum_{j=0}^{r-1}t^{j}(1+t^2+\cdots +t^{2(r-j-1)})\sfh_{(n-j,j)}.
\eeq
	Applying the Stanley map $\stan$, we obtain the characteristic polynomial $\stan_{Y_n}(m,t)$ of $Y_n$ as follows. 
\begin{proposition} For $n=2r+1$, 
    \[\stan_{Y_n}(m,t)
    =\sum_{j=0}^{r-1}t^j(1+t^2+\cdots +t^{2(r-j-1)})\multisetBig{m}{n-j}\multisetBig{m}{j}.\]
    In particular, 
    \[\begin{split}
    	&\stan_{Y_n}(m,1)=\sum_{j=0}^{r-1}(r-j)\multisetBig{m}{n-j}\multisetBig{m}{j}, ~\text{ and} \\
    	&\stan_{Y_n}(1,t)=\sum_{j=0}^{r-1}t^j(1+t^2+\cdots +t^{2(r-j-1)}).
    \end{split}
     \]
\end{proposition}
\begin{proof}
	This is immediate from \eqref{eq:GIT}, since $\stan_{\sfh_{(a,b)}}=\multisetbody{m}{a}\multisetbody{m}{b}$.
\end{proof}
Numerical calculations lead us to the following conjecture.
\begin{conjecture}\label{y23}
Let $n\geq 3$ be an odd integer.
	\begin{enumerate}
		\item $\stan_{Y_n}(m,t)$ is length log-concave in degree $k$ for any $0\leq k\leq n$.
		\item $\stan_{Y_n}(m,t)$ is length log-concave at $t=1$, i.e. $\stan_{Y_n}(m,1)\in \Q[m]$ is log-concave. 
	\end{enumerate}
\end{conjecture}
	We have checked that (1) and (2) hold for $n\leq 400$.

\begin{example}
	When $n=5$, the graded $\symS_n$-representation on $H^*(Y_5)$ is
	$\sum_{k=0}^2\ch_{\symS_5}H^{2k}(Y_5)t^k=\sfh_5(1+t^2)+\sfh_{(4,1)}t$. 
	Hence by applying $\stan$, we obtain
	\[\begin{split}
		\stan_{Y_5}(m,t)&=\frac{1}{120}(m^5+10m^4+35m^3+50m^2+24m)(1+t^2)\\
		&\quad +\frac{1}{24}(m^5+6m^4+11m^3+6m^2)t.
	\end{split}\]
	Evaluating it at $t=1$, we get
	\[\begin{split}
		\stan_{Y_5}(m,1)
		&=\frac{1}{120}(7m^5+50m^4+125m^3+130m^2+48).
	\end{split}
	\]
	It is easy to see that Conjecture \ref{y23} holds in this case.
\end{example}

On the other hand, the degree log-concavities fail for $Y_n$. More precisely, $\stan_{Y_n}(m,t)$ is degree log-concave neither  at $m=1$ nor in length 1.
\begin{example}
	$\stan_{Y_n}(1,t)$ is not log-concave whenever $n\geq7$. Indeed,
	\[\begin{split}
		\stan_{Y_7}(1,t)=1+t&+2t^2+t^3+t^4,\\
		\stan_{Y_9}(1,t)=1+t+2t^2&+2t^3+2t^4+t^5+t^6,\\
		\stan_{Y_{11}}(1,t)=1+t+2t^2+2t^3&+3t^4+2t^5+2t^6+t^7+t^8,\\
		&\quad\cdots
	\end{split}\]
	
	One can also easily check that the coefficient of $m$ in $\stan_{Y_n}(m,t)$ is
	\[\frac{1}{n}\sum_{j=0}^{r-1}t^{2j}=\frac{1}{n}(1+t^2+\cdots + t^{2r-4}+t^{2r-2})\]
	which is not log-concave due to the internal zeros at odd degrees.

\end{example}

Interestingly, numerical calculation shows that $\stan_{Y_n}(m,t)$ is degree log-concave in length $\ell$ when $\ell$ is \emph{sufficiently close} to $n$. For example, $\stan_{Y_n}(m,t)$ is degree log-concave in length $\ell$ for any $\ell$ in the range $\sqrt{n}+2\leq \ell \leq n$ for $5\leq n\leq 400$, although the bound $\sqrt{n}+2$ is definitely not sharp. Computing a sharp bound and analyzing its asymptotic behavior as $n$ grows would also be of interest.

\medskip

\subsection{Hessenberg varieties}\label{ss:Hess}
Regular semisimple Hessenberg varieties are subvarieties in the flag variety, naturally equipped with an $\symS_n$-action on its cohomology. For further details on the discussion below, we refer the reader to a survey \cite{AH}.

Denote the variety of flags in $\C^n$ by $\mathrm{Fl}_n=\{F_\bullet=(F_1\subset \cdots \subset F_n=\C^n)\}$ where $F_i$ are $i$-dimensional subspaces of $\C^n$ for each $i\in [n]:=\{1,\cdots,n\}$.

A \emph{Hessenberg function} is a non-decreasing function $h:[n]\to [n]$ with $h(i)\geq i$ for all $i$. Choosing a regular semisimple $n\times n$ matrix $s$ (or equivalently, an $n\times n$ matrix with $n$ distinct eigenvalues), one obtains a subvariety in $\mathrm{Fl}_n$, called a \emph{(regular semisimple) Hessenberg variety}, defined as
\[X_h=\{F_\bullet \in \mathrm{Fl}_n: sF_i\subset F_{h(i)} ~\text{ for all }i\}.\]
Here, we suppress $s$ in the notation, as its diffeomorphism type (hence its cohomology) does not depend on the choice of $s$. 

These varieties were introduced in \cite{DMPS}, where it is also proved that they are smooth of dimension $\sum_{i=1}^n(h(i)-i)$ when $s$ is regular semisimple. In this case, $X_h$ admits a natural action of the maximal torus $T$ in $\GL_n(\C)$ centralizing $s$. Since its $T$-fixed locus is finite (and isomorphic to $\symS_n$), 
$X_h$ admits a cell decomposition by the Bia\l{}ynicki-Birula theorem and $H^{2k+1}(X_h)=0$ for all $k$. 
Later in \cite{Tym}, Tymoczko defined an $\symS_n$-action on the $T$-equivariant cohomology $H^*_T(X_h)$ using the GKM theory \cite{GKM} and $X_h^T\cong \symS_n$, and showed that it descends to an $\symS_n$-action on $H^*(X_h)$ via the canonical projection $H^*_T(X_h)\to H^*(X_h)$. 
We note that this $\symS_n$-action on $H^*(X_h)$ may not lift to an $\symS_n$-action on $X_h$ \cite{BEHLLMS}.

We thus obtain 
\beq\label{y27} \sum_{k\ge 0}\ch_{\symS_n} (H^{2k}(X_h))t^k\in \Lambda_n[t]\eeq
whose characteristic polynomial is 
\beq\label{y25}\stan_{X_h}(m,t)=
\sum_{k\ge 0}\stan_{H^{2k}(X_h)}(m)t^k\in \Q[m,t].\eeq 

\smallskip
On the other hand, there is a purely combinatorial way to associate a symmetric function to each Hessenberg function $h$. 
The chromatic symmetric function introduced by Stanley in \cite{Sta2} associates a symmetric function 
\beq\label{eq:csf}\csf_{\Gamma}=\sum_{\kappa:[n]\to \Z_{>0}\text{ proper}}x_{\kappa(1)}\cdots x_{\kappa(n)} \in \Lambda_n\eeq 
to each graph $\Gamma$ with $n$ vertices, where $\kappa$ runs over all the coloring maps from the vertex set $[n]$ to the set of colors indexed by positive integers that are \emph{proper}, that is, $\kappa(i)\neq \kappa(j)$ whenever $(i,j)\in E(\Gamma)$. By the substitution \eqref{4}, its characteristic polynomial 
$$\stan_{\csf_{\Gamma}}(m)=\chrom_{\Gamma}(m)$$
is the usual chromatic polynomial. 

In \cite{SW}, Shareshian and Wachs introduced the notion of the chromatic \emph{quasi}symmetric function $\cqsf_\Gamma$ of a graph $\Gamma$, which is a polynomial in a variable $t$ whose coefficients are quasisymmetric functions, that is, functions invariant under order-preserving change of variable indices. The chromatic quasisymmetric function reduces to the original chromatic symmetric function when evaluated at $t=1$. 

For a Hessenberg function $h:[n]\to [n]$, 
let $\Gamma_h$ be the \emph{incomparability graph associated to $h$}, i.e. the graph whose vertex set and edge set are 
\[V(\Gamma_h)=[n] \and E(\Gamma_h)=\{(i,j)\in [n]^2:i<h(i)\leq j\}\]
respectively.
Shareshian and Wachs in \cite{SW} showed that the chromatic quasisymmetric function of the graph $\Gamma_h$
is in fact a graded symmetric function  
\beq\label{y28} \cqsf_{\Gamma_h}(t)\in \Lambda_n[t]\eeq   
so that we have the characteristic polynomial 
\beq\label{y26} \stan_{\Gamma_h}(m,t):=\stan_{\cqsf_{\Gamma_h}}(m,t) \in \Q[m,t]
\eeq
that satisfies 
\beq\label{y31}
\stan_{\Gamma_h}(m,1)=\stan_{\csf_{\Gamma_h}}(m)=\chrom_{\Gamma_h}(m).\eeq

\smallskip

A remarkable conjecture by Shareshian-Wachs  in \cite{SW}, proved by Brosnan-Chow  and independently by Guay-Paquet tells us that \eqref{y28} is the $\omega$-dual of \eqref{y27} where $\omega:\Lambda\to \Lambda$ is the involution sending $\sfe_\lambda$ to $\sfh_\lambda$. 
\begin{theorem}\cite{BC, GP2}\label{thm:SWconj} For any Hessenberg function $h$,
$$\omega\sum_{k\ge 0}\ch_{\symS_n}(H^{2k}(X_h))t^k=\cqsf_{\Gamma_h}.$$ 
\end{theorem}
See \cite{KL} for an elementary proof based on the modular law.

By applying the Stanley map $\stan$, we get the following.
\begin{corollary}\label{y30}
Let 
\[\stanomega_{Z}(m,t)=\stan\left(\omega \sum_{k\ge 0}\ch_{\symS_n}(H^{2k}(Z))t^k\right)=\sum_{k\ge 0}\stan_{H^{2k}(Z)\otimes\sgn_n}t^k\]
for any connected smooth projective variety $Z$ with an $\symS_n$-action on each even degree cohomology group of $Z$. Then 
\beq\label{y29} \stanomega_{X_h}(m,t)=\stan_{\Gamma_h}(m,t).\eeq
\end{corollary}
By \eqref{y31} and \eqref{y29}, we have 
\beq\label{eq:chrom.poly.Hess}\stanomega_{X_h}(m,1)=\chrom_{\Gamma_h}(m)=\prod_{i=1}^n\left(m-(h(i)-i)\right),\eeq
where the second equality is an easy calculation based on the fact that $\Gamma_h$ is directed. 
In particular, this is log-concave as it should be  by a general result of Huh \cite{Huh}. 

\smallskip
By Lemma~\ref{l:omega}, we have 
\beq\label{y33}\stanomega_{F}(m)=(-1)^n\stan_F(-m)\eeq for any $F\in \Lambda_n$ and the following is immediate.
\begin{proposition}
	The characteristic polynomial 
	of  a Hessenberg variety $X_h$ satisfies 
	\beq\label{eq:char.poly.Hess}\stan_{X_h}(m,1)=\prod_{i=1}^n (m+h(i)-i)\eeq
	and $\stan_{X_h}(m,t)$ is length log-concave at $t=1$.
\end{proposition}

There are combinatorial formulas for the coefficients of $\sfp_\lambda$ in $\cqsf_{\Gamma_h}(t)$ in \cite[Theorem 4 and Corollary~7]{Ath}. These immediately induce combinatorial formulas for $\stan_{X_h}(m,t)$ by \eqref{y29} and \eqref{y33}.
We do not know yet how to deduce or disprove the other log-concavities from them, except the following two results.
Write $[r]_t:=\sum_{i=0}^{r-1}t^i$ for $r\geq 0$.
\begin{enumerate}
	\item The evaluation of $\stan_{X_h}(m,t)$ at $m=1$ is 
	\[\stan_{X_h}(1,t)=\sum_{k\geq0}\dim (H^{2k}(X_h)^{\symS_n})t^k=\prod_{i=1}^n[h(i)-i+1]_t.\]
	In particular, $\stan_{X_h}(m,t)$ is degree log-concave at $m=1$.
	\item The coefficient of $m$ in $\stan_{X_h}(m,t)$ is 
	\[
    \frac{1}{n}[n]_t\prod_{i=1}^{n-1}[h(i)-i]_t.\]
    In particular, $\stan_{X_h}(m,t)$ is degree log-concave at length 1.
\end{enumerate}
Both formulas follow by combining the results of \cite{SW,AHHM,ST}, together with Lemmas~\ref{l:mult} and \ref{l:M.mult}. See also \cite[Proposition~5.10]{KL}.

\medskip
 
We end this section with the following natural question which may be regarded as a graded version of Huh's theorem in \cite{Huh}.

\begin{question}
	Is $\stan_{X_h}(m,t)$ or $\stanomega_{X_h}(m,t)=\stan_{\Gamma_h}(m,t)=\stan_{\cqsf_{\Gamma_h}}(m,t)$ length log-concave in degree $k$ for any $k$?
\end{question}

\bigskip

\section{Characteristic polynomial of $\Mbar_{0,n}$}\label{s:M0n}
The moduli space $\Mbar_{0,n}$ of $n$-pointed stable curves of genus $0$ is one of the most important varieties in algebraic geometry and has been much studied (cf.~\cite[\S1]{CKL}). It comes with a natural action of the symmetric group $\symS_n$ by permuting the marked points and thus its cohomology $H^*(\Mbar_{0,n})$ is a graded $\symS_n$-module. The rest of this paper is devoted to the study of the characteristic polynomial $\stan_{\Mbar_{0,n}}$ of $\Mbar_{0,n}$ as well as that of $\Mbar_{0,n+1}$ on which $\symS_n$ is acting by permuting the first $n$ marked points with the last one fixed. Based on our pervious results in \cite{CKL, CKL2}, we provide recursive formulas for the characteristic polynomials (Theorem~\ref{thm:recursion}).
In the last subsection \S\ref{ss:coloring.UT}, we find an interpretation of the characteristic polynomials of permutation representations that constitute the $\symS_n$-representation $H^*(\Mbar_{0,n+1})$, in terms of colorings (Corollary~\ref{cor:coloring.chiUT}).

\subsection{Recursive formulas for the representations}\label{ss:recursion.rep}
This section reviews the results in \cite{CKL,CKL2}, with a focus on formulas. For an overview together with related geometry, we refer the reader to \cite[\S3]{CKL2}. 

The moduli space $\Mbar_{0,n}$ is a compactification of the variety $$(\PP^1)^n-\Delta/\Aut(\PP^1)$$
of $n$ distinct ordered points on $\PP^1$ up to the automorphism group action. 
It certainly comes with the natural action of $\symS_n$ permuting the marked points. 
When studying the cohomology of $\Mbar_{0,n}$ as a graded $\symS_n$-module, it helps to study the cohomology of $\Mbar_{0,n+1}$ as well, on which $\symS_n$ is acting by permuting the first $n$ marked points with the last one fixed.  

We write the $\symS_n$-equivariant Poincar\'e polynomial of $\Mbar_{0,n}$ and $\Mbar_{0,n+1}$ as
\[\begin{split}
	&P_n=\sum_{k=0}^{n-3}P_{n,k}t^k \in \Lambda_n[t], \quad \text{where }~P_{n,k}=\ch_{\symS_n}H^{2k}(\Mbar_{0,n})\in \Lambda_n;\\
	&Q_n=\sum_{k=0}^{n-2}Q_{n,k}t^k \in \Lambda_n[t], \quad \text{where }~Q_{n,k}=\ch_{\symS_n}H^{2k}(\Mbar_{0,n+1})\in \Lambda_n.
\end{split}\]
Moreover, we consider their generating series as
\[\begin{split}
	&P=1+\sfh_1+\sfh_2+\sum_{n\geq 3}P_n\quad \in \Lambda\llbracket t\rrbracket=\prod_n\Lambda_n\llbracket t\rrbracket;\\
	&Q=1+\sfh_1+\sum_{n\geq 2}Q_n\quad \in \Lambda\llbracket t\rrbracket,
\end{split}
\]
where we treat $\Mbar_{0,n}$ as a point when $0\leq n\leq 2$.

In \cite{CKL, CKL2}, we established formulas for $P$ and $Q$ which were obtained in two stages:
(1) comparison of $P$ and $Q$, (2) computation of $Q$.

\noindent\underline{Stage 1}. 
The relation between $P_n$ and $Q_n$ is obtained in \cite[\S4]{CKL} by the theory of $\delta$-stable quasimaps, developed in \cite{CK}, in a very special case which provides an $\symS_n$-equivariant factorization of the canonical forgetful morphism $\Mbar_{0,n+1}\to \Mbar_{0,n}$ by blowups. 
By the blowup formula and geometric invariant theory, we have the following.  
\begin{theorem}\cite[Theorem~4.8]{CKL} \label{thm:CKL.quasimap}
	For $n\geq 3$, 
	\[(1+t)P_n=Q_n-t\left(\sum_{2\leq h<\frac{n}{2}}Q_hQ_{n-h}+\sfs_{(1,1)}\circ Q_{\frac{n}{2}}\right)\]
	where we set $Q_{\frac{n}{2}}=0$ for $n$ odd. Equivalently, the generating series $P$ and $Q$ satisfy (cf.~\cite[Corollary~3.3]{CKL2})
	\[(1+t)P=(1+t+\sfh_1t)Q-t\sfs_{(1,1)}\circ Q.\]
\end{theorem}

\noindent\underline{Stage 2}. 
The combinatorial formula for $Q_n$ in \cite{CKL} is derived from the well-known $\symS_n$-equivariant factorization of the Kapranov map 
$$\Mbar_{0,n+1}\lra \PP^{n-2}$$ by a sequence of blowups. This corresponds to a realization of $\Mbar_{0,n+1}$ as the wonderful compactification of the complement of the braid hyperplane arrangement 
\cite{Kap}. 
The challenging task of tracking all the blowup centers is accomplished by the combinatorics of weighted rooted trees.

\begin{definition}\cite[Definition~3.5]{CKL2}
	A \emph{labeled rooted tree} 
	is a tree with one distinguished leg called, the \emph{output}, and additional $n$ legs, called the \emph{inputs}, labeled by $1,\cdots,n$, for some $n\in \Z_{\geq0}$. The vertex to which the output is attached is called the \emph{root}. 
    
    More concretely, a labeled rooted tree is a triple of a tree $T$, the output map $\mathrm{pt}\to V(T)$ and the input map $\{1,\cdots, n\}\to V(T)$, where $V(T)$ is the set of vertices of $T$. In particular, permuting the labels of the inputs attached to the same vertex does not distinguish the labeled rooted tree. 
    
    A \emph{rooted tree} is a tree obtained from a labeled rooted tree by forgetting the labels on the inputs. 
\end{definition}
\begin{definition}
	Let $T$ be a (resp. labeled) rooted tree. A \emph{weight function} on $T$ is a  function $w:V(T)\to \Z_{\geq0}$ on the vertex set $V(T)$. We denote the total weight by $\lvert w\rvert=\sum_{v\in V(T)}w(v)$. 
	
	We call the pair $(T,w)$ a (resp. \emph{labeled}) \emph{weighted rooted tree}.
\end{definition}

\def\val{\mathrm{val} }
\begin{definition}\label{y20}
	For $n$ and $k$ in $\Z_{\geq0}$, we define $\sT_{n,k}$ (resp. $\sT_{n,k}^{\mathrm{lab}}$) to be the set of 
	(resp. labeled) weighted rooted trees $(T,w)$
	with $n$ inputs and 
	total weight $\lvert w\rvert =k$, that satisfy the following \emph{valency conditions}:
	\begin{enumerate}
		\item $0\leq w(v)\leq \val(v)-3$ if $v$ is the root;
		\item $0< w(v) \leq \val(v)-3$ if $v$ is not the root.
	\end{enumerate}
	There is a natural $\symS_n$-action on $\sT_{n,k}^{\mathrm{lab}}$ reordering labels and $\sT_{n,k}=\sT_{n,k}^{\mathrm{lab}}/\symS_n$.
\end{definition}
\begin{proposition}\cite[Proposition~5.12]{CKL}\label{p:CKL.Q}
	The $\symS_n$-representation $H^{2k}(\Mbar_{0,n+1})$ is a permutation representation with an $\symS_n$-basis isomorphic to $\sT_{n,k}^{\mathrm{lab}}$.
\end{proposition}
We decompose this permutation representation into transitive permutation representations $U_H$ in Example \ref{ex:chi.U_H} as follows.
\begin{definition}
	For $(T,w)\in \sT_{n,k}^{\mathrm{lab}}$, we denote by $\Stab(T,w)$ the stabilizer subgroup of $(T,w)$ in $\symS_n$, that is the subgroup consisting of permutations of labels that preserves the isomorphism class of $(T,w)$ as a labeled weighted rooted tree.   We define 
	\beq\label{y18} U_{(T,w)}=U_{\Stab(T,w)}\eeq
	using  the notation of \eqref{y16}.
	For $(T,w)\in \sT_{n,k}=\sT_{n,k}^{\mathrm{lab}}/\symS_n$, we define $\Stab(T,w)$ and $U_{(T,w)}$ similarly by taking any lift in $\sT_{n,k}^{\mathrm{lab}}$. 
\end{definition}
Note that the subgroup $\Stab(T,w)\subset \symS_n$ is well defined up to conjugation and $U_{(T,w)}$ is well defined as an $\symS_n$-representation.

For convenience, we will sometimes omit $w$ in our notation, and write $T$, $\Stab(T)$ and $U_T$ for $(T,w)$, $\Stab(T,w)$ and $U_{(T,w)}$ respectively.

\begin{example}\label{ex:stab1}
	Suppose that $T$ has only one vertex, namely the root, and all the $n$ inputs are attached to the root. Then $\Stab(T)\cong\symS_n$.
\end{example}

\begin{example}\label{ex:stab2}
	Suppose that $T$ has only two vertices with $a$ inputs attached to the root and the remaining $n-a$ inputs attached to the other non-root vertex. Then, $\Stab(T)\cong \symS_a\times \symS_{n-a}$.
\end{example}
\begin{example}\label{ex:stab3}
	Suppose that $T$ has three vertices where both non-root vertices are adjacent to the root and have weight one. Suppose further that to each non-root vertex, $a$ inputs are attached and the remaining $n-2a$ vertices are attached to the root. Then, $\Stab(T)\cong \symS_{n-2a}\times((\symS_a\times \symS_a)\rtimes \symS_2)$.
\end{example}

Proposition~\ref{p:CKL.Q} now reads as
\[ Q_{n,k}=\sum_{T\in \sT_{n,k}}\ch_{\symS_n}U_{T}.\]
Combining Theorem~\ref{thm:CKL.quasimap} and Proposition~\ref{p:CKL.Q}, we obtain formulas computing $P$ and $Q$ in terms of weighted rooted trees.

For actual computations however, the combinatorial complexity of weighted rooted trees grows quickly as $n$ increases. As a remedy, in \cite{CKL2}, a natural recursive structure on $\sT_{n,k}$ is observed, which leads to recursive formulas for computing $Q$ and $P$. 

Let $\sT_{n,k}^+\subset \sT_{n,k}$ be the subset consisting of weighted rooted trees that have positive weights at the root. Let $Q_1^+=h_1$ and for $n\ge 2$, 
\[Q_n^+=\sum_{k=0}^{n-2}Q_{n,k}^+t^k, \quad \text{ where }~Q_{n,k}^+=\sum_{T\in \sT_{n,k}^+}\ch_{\symS_n}U_T.\]
Consider its generating series 
\[Q^+=Q_1^++\sum_{n\geq2}Q_n^+ \quad \in \Lambda\llbracket t\rrbracket.\]

\begin{theorem}\cite[Theorem~4.7]{CKL2}\label{thm:CKL.recursive}
	$Q$ is the plethystic exponential of $Q^+$: 
	\[Q=\Exp (Q^+)=\sum_{r\geq0}\sfh_r\circ Q^+.\]
Consequently, $Q^+$ is determined by the formulas (cf.~\cite[Corollary~4.9]{CKL2}) 
	\[
	\begin{split}
		\mathrm{(Recursive)} \quad &Q^+=\sfh_1 +\sum_{r\geq 3}\left(\sum_{i=1}^{r-2}t^i\right)(\sfh_r\circ Q^+),\\
		\mathrm{(Exponential)} \quad &\Exp(tQ^+)=t^2\Exp(Q^+)+(1-t)(1+t+\sfh_1t ).
	\end{split}
	\]
More explicitly, $Q_n$ and $Q_n^+$ are determined by 
	\[\begin{split}
	    Q_n^+&=
	    \sum_{\lambda\vdash n}
	    \left(\sum_{i=1}^{\ell(\lambda)-2}t^i\right)\prod_{j=1}^m\left(\sfh_{r_j}\circ Q_{n_j}^+\right) \and
		Q_n=
		\sum_{\lambda\vdash n}
		\prod_{j=1}^m\left(\sfh_{r_j}\circ Q_{n_j}^+\right)
	\end{split}\]
for $n\geq2$,	where  $n_j$ denote the parts of $\lambda$ with multiplicities $r_j$ so that $\lambda=(n_1^{r_1},\cdots, n_m^{r_m})$ with $n_1>\cdots >n_m>0$
 and $\ell(\lambda)=\sum_{j=1}^mr_j$. 
\end{theorem}

In the next subsection, we will apply the Stanley map $\widehat \stan$ to the formulas in Theorems~\ref{thm:CKL.quasimap} and \ref{thm:CKL.recursive} to deduce corresponding formulas for the characteristic polynomials of $\Mbar_{0,n}$ and $\Mbar_{0,n+1}$. 
Since every formula appearing above involves only arithmetic operations and plethysm, our formulas for the characteristic polynomials involve only arithmetic operations and composition. 

The proof of Theorem~\ref{thm:CKL.recursive} uses the recursive structure on $\sT_{n,k}$ which will be briefly explained in Lemma~\ref{lem:decomp.wrt} in \S\ref{ss:coloring.UT} below.

\medskip

\subsection{Recursive formulas for characteristic polynomials}\label{ss:recursion.char.poly}
Let us apply the Stanley map to $P, P_n, P_{n,k}, Q, Q_n, Q_{n,k}$ to obtain the characteristic polynomials of $\Mbar_{0,n}$ and $\Mbar_{0,n+1}$ as 
\[\begin{split}
	(\charP,\charP_n,\charP_{n,k})&:=(\widehat\stan_P,\stan_{P_n},\stan_{P_{n,k}});\\
	(\charQ,\charQ_n,\charQ_{n,k})&:=(\widehat\stan_{Q},\stan_{Q_n},\stan_{Q_{n,k}});\\
	(\charQ^+,\charQ_n^+,\charQ_{n,k}^+)&:=(\widehat\stan_{Q^+},\stan_{Q_n^+},\stan_{Q_{n,k}^+}).
\end{split}
\]
In particular, the generating series are written as
\[\begin{split}
	&\charP(m,q,t):=1+mq+\frac{m(m+1)}{2}q^2+\sum_{n\geq2}\charP_{n}(m,t)q^n;\\
	&\charQ(m,q,t):=1+mq+\sum_{n\geq2}\charQ_{n}(m,t) q^n;\\
	&\charQ^+(m,q,t):=mq+\sum_{n\geq2}\charQ^+_{n}(m,t)q^n.
\end{split}
\]

The following are direct consequences of Theorems~\ref{thm:CKL.quasimap} and \ref{thm:CKL.recursive}, thanks to the properties of $\stan$ and $\widehat\stan$ in \eqref{1}, \eqref{y11} and \eqref{y11a}.

Let $\Par$ be the set of all the partitions of nonnegative integers as in Remark~\ref{ex:Exp}. 
We write $\lvert \lambda\rvert =\sum_{i=1}^\ell \lambda_i$ for $\lambda=(\lambda_1,\cdots, \lambda_\ell)$.
Then $\charP$ and $\charQ$ are determined by the following.
\begin{theorem}\label{thm:recursion}
\begin{enumerate}
	\item $\charP$ and $\charQ$ are related by the formula
	\[
	(1+t)\charP=(1+t+mqt)\charQ- \frac{t}{2}(\charQ^2-\charQ^{[2]}). 
    \]
	\item $\charQ$ is the plethystic exponential of $\charQ^+$: \[\charQ=\Exp(\charQ^+)=\sum_{\lambda\in \Par}\centz_\lambda^{-1}(\charQ^+)^{[\lambda]}.\]
	\item Moreover, $\charQ^+$ satisfies the following two equivalent formulas.
	\[\begin{split}
		\mathrm{(Recursive)} \quad &\charQ^+=mq +
        \sum_{\lambda\in \Par, \lvert \lambda\rvert \geq 3}\left(\sum_{i=1}^{\lvert \lambda\rvert-2}t^i\right)\centz_\lambda^{-1}(\charQ^+)^{[\lambda]}\\
		\mathrm{(Exponential)} \quad &\Exp(t\charQ^+)=t^2\Exp(\charQ^+)+(1-t)(1+t+mqt ).
	\end{split}\]
\end{enumerate}
\end{theorem}
\begin{proof}
	For (1), it suffices to show that $\sfs_{(1,1)}\circ \charQ=\frac{1}{2}(\charQ^2-\charQ^{[2]})$. This follows from the identity $\sfs_{(1,1)}=\frac{1}{2}(\sfp_1^2-\sfp_2)$. For (2) and (3), see Remark~\ref{ex:Exp}.
\end{proof}

\begin{corollary}\label{cor:recursion}
		For $n\geq 3$, $\charP_n$ and $\charQ_n$ are related by
	\beq\label{y17} (1+t)\charP_n=\charQ_n-\frac{1}{2}t\left(\sum_{h=2}^{n-2}\charQ_h\charQ_{n-h}-\charQ_{\frac{n}{2}}^{[2]}
    \right),\eeq 
	where $\charQ_\frac{n}{2}=0$ for $n$ odd.
	Set $\charQ_1^+=m$. For $n\geq2$, $\charQ_n^+$ and $\charQ_n$ satisfy 
	\[\begin{split}
	    \charQ_n^+&=\sum_{\lambda\vdash n}\left(\sum_{i=1}^{\ell(\lambda)-2}t^i\right)\prod_{j=1}^m \left(\sfh_{r_j}\circ\charQ_{n_j}^+\right) \and
		\charQ_n=\sum_{\lambda\vdash n}\prod_{j=1}^m \left(\sfh_{r_j}\circ\charQ_{n_j}^+\right)
	\end{split}\]
	where $n_j$ denote the parts of $\lambda$ with multiplicities $r_j$ so that $\lambda=(n_1^{r_1},\cdots, n_m^{r_m})$ with $n_1>\cdots >n_m>0$
 and $\ell(\lambda)=\sum_{j=1}^mr_j$.

\end{corollary}

\begin{remark}\label{rem:plethysm}
    Note that $\sfh_r\circ \charQ_n^+=\sum_{\lambda\vdash r}\centz_\lambda^{-1}(\charQ_n^+)^{[\lambda]}$ for $r,n\geq1$, by the same argument as in Remark~\ref{ex:Exp}.
\end{remark}

For later use, we record a different but equivalent form of \eqref{y17}.
For any integer $k$, reading the coefficients of $t^k$ in \eqref{y17}, we obtain
\beq\label{eq:walcrossing.charpoly}
	\begin{split}
		\charP_{n,k}+\charP_{n,k-1}=\charQ_{n,k}&-\frac{1}{2}\sum_{2\leq h\leq n-2}\sum_{j=0}^{k-1}\charQ_{h,j}\charQ_{n-h,k-j-1}+ \frac{1}{2}\charQ_{\frac{n}{2},\frac{k-1}{2}}
	\end{split}
\eeq
where we set $\charQ_{h,j}=0$ whenever $h$ or $j$ is not an integer.

\medskip

Using the algorithm in this section, we can study the log-concavities of $\charP_n(m,t)$ and $\charQ_n(m,t)$ (cf.~Definition \ref{def:lc}). 
Refining Definition~\ref{def:lc}(2), we say that a polynomial $A(m,t)$ in~\eqref{y0} is length log-concave in degree $k$ \emph{at length $\ell$} 
if $A_k(m)=\sum_{j}A_k^jm^j$ is log-concave at $j=\ell$, that is, $(A^{\ell}_k)^2\geq A^{\ell-1}_kA^{\ell+1}_k$.

Based on numerical computations, we propose the following.

\begin{conjecture}\label{conj} The polynomials $\charP_{n}(m,t)$ and $\charQ_{n}(m,t)$ are
\begin{enumerate}
\item length log-concave at $t=1$, 
\item length log-concave in degree $k$ (for any given $k$) at length $\ell>2$, 
\item degree log-concave at $m=1$ and 
\item degree log-concave in length $\ell $ for $\ell \ge \sqrt{n} +2$.
\end{enumerate}
\end{conjecture}

We verified this conjecture up to $n\le 72$. 
In \S6, we will prove (3) and (4) for $n$ sufficiently large as well as (2) partially.

\begin{remark}
(1)	Based on explicit computations, we observe that Conjecture~\ref{conj}(2) does \emph{not} hold in general for $\ell=2$. The following is the complete list of pairs $(n,k)$ with $n\leq 72$ for which the polynomials $\charP_n(m,t)$ and $\charQ_n(m,t)$ fail to be length log-concave in degree $k$ at length $\ell=2$.
	\begin{itemize}
  \item For $\charP_n(m,t)$:
    \begin{itemize}
      \item $n=35$ and $k=7, 26$;
      \item $n=60$ and $k=25,32$;
      \item $n=63$ and $k=9,11,49,51$;
      \item $n=72$ and $k=27,30,32,33,36,37,39,42$.
    \end{itemize}
  \item For $\charQ_n(m,t)$:
    \begin{itemize}
      \item $n=72$ and $k=28,30,40,42$.
    \end{itemize}
\end{itemize}

(2) It is also observed that, regarding Conjecture~\ref{conj}(4), $\charP_n(m,t)$ and $\charQ_n(m,t)$ are \emph{not} degree log-concave in all lengths in general. For example, since $P_5=\sfh_5 + \sfh_4\sfh_1 t + \sfh_5t^2$, we have 
    \begin{align*}
        \charP_5(m,t) &= \frac{m^5}{120}+\frac{m^4}{12}+\frac{7m^3}{24}+\frac{5m^2}{12}+\frac{m}{5} \\
        &+\left(\frac{m^5}{24}+\frac{m^4}{4}+\frac{11 m^3}{24}+\frac{m^2}{4}\right)t \\
        &+\left(\frac{m^5}{120}+\frac{m^4}{12}+\frac{7 m^3}{24}+\frac{5m^2}{12}+\frac{m}{5}\right)t^2, 
    \end{align*}
    which is not degree log-concave in length 1 and 2. 
    
    Furthermore, when $n=6$, applying \eqref{eq:qnk1} yields the sequence 
    \[\{\charQ_{6, k}^1\}_{k=0}^4 = \left\{\frac{1}{6},\frac{1}{6},\frac{1}{3},\frac{1}{6},\frac{1}{6}\right\}.\]
    Hence, $\charQ_6$ is not degree log-concave in length 1. We note that the bound $\sqrt{n} +2$ in Conjecture~\ref{conj}(4) is not sharp.
\end{remark}

In \S\ref{s:asymp}, we will use Theorem~\ref{thm:recursion} and Corollary~\ref{cor:recursion} to deduce asymptotic formulas and asymptotic log-concavities for the characteristic polynomials of $\Mbar_{0,n}$ and $\Mbar_{0,n+1}$.

\medskip
\subsection{Interpretation as the count of colorings}\label{ss:coloring.UT}
In this subsection, we present an interpretation of the characteristic polynomial $\stan_{U_T}(m)$ of \eqref{y18} and its chromatic polynomial $\stanomega_{U_T}(m)$ in terms of colorings, which may be of independent interest.  

We first recall an inductive decomposition of weighted rooted trees, based on which the recursion in Theorem~\ref{thm:CKL.recursive} was established in \cite{CKL2}.

\begin{lemma} \cite[Lemma~4.4]{CKL2}\label{lem:decomp.wrt}
	Let $T\in \sT_{n,k}$. Let $v_0$ denote the root vertex of $T$. Let $n_0$ and $k_0$ be the number of the inputs attached to $v_0$ and the weight given at $v_0$ respectively. 
	Let $v_1,\cdots, v_p$ be the non-root vertices of $T$ adjacent to $v_0$. For $1\leq i\leq p$, let $T_i'$ denote the weighted rooted subtree of $T$ such that 
	\begin{enumerate}
		\item its underlying rooted subtree has $v_i$ as its root;
		\item $V(T_i')$ consists of $u\in V(T)$ such that $u$ and $v_i$ are connected by a path in $T$ which does not pass through $v_0$;
		\item the edges (resp. inputs) of $T_i'$ are precisely those of $T$ connecting (resp. attached to) the vertices in $V(T_i')\subset V(T)$;
		\item the weight function on $V(T_i)$ is the restriction of that of $T$.
	\end{enumerate}
	 Let $T_0\in \sT_{n_0,k_0}$ be the weighted rooted tree with only one vertex. Now let
	 \[\{T_1',\cdots, T_p'\}=\{T_1,\cdots, T_1,\cdots, T_l,\cdots, T_l\}\]
	 where $T_j$ are mutually non-isomorphic and there are $r_j$ copies of $T_j$. Assume that $T_{j}\in \sT_{n_j,k_j}^+$ so that $n=n_0+\sum_{j=1}^{l}r_jn_j$ and $k=k_0+\sum_{j=1}^l r_jk_j$.
     Then, 
	 \beq\label{ex:chUT}
	 	\ch_{\symS_n}U_T=\sfh_{n_0}\cdot \prod_{j=1}^l\sfh_{r_j}\circ \ch_{\symS_{n_j}}(U_{T_j}).\eeq
    Furthermore, by taking the involution $\omega$, we have 
      \beq\label{ex:omegachUT}
	 	\omega(\ch_{\symS_n}U_T)=\sfe_{n_0} \prod_{r_j:\;\mathrm{even}}\sfh_{r_j}\circ \omega(\ch_{\symS_{n_j}}U_{T_j})\prod_{r_j: \;\mathrm{odd}}\sfe_{r_j}\circ \omega(\ch_{\symS_{n_j}}U_{T_j}).
	 \eeq
\end{lemma}
Pictorially, $T$ decomposes as follows.
	 \[\begin{tikzpicture} [scale=.7,auto=left,every node/.style={scale=0.8}]
      \tikzset{Bullet/.style={circle,draw,fill=black,scale=0.5}}
      \node[Bullet] (T0) at (-1,-1) {};
      \node[Bullet] (T1) at (0,-1) {};
      \node[Bullet] (T2) at (3,-1) {};
      \node[Bullet] (n0) at (-6,-0.5) {};
      \node[Bullet] (n1) at (-7,-1.5) {};
      \node[Bullet] (n2) at (-5,-1.5) {};

      \draw[black] (1.5,-1) node[] {$\cdots$};
      \draw[black] (T0) -- (-1,0);
      \draw[black] (T1) -- (0,0);
      \draw[black] (T2) -- (3,0);
      \draw[black] (-1,-1.2) node[below] {$T_0$};
      \draw[black] (0,-1.2) node[below] {$T_1$};
      \draw[black] (3,-1.2) node[below] {$T_l$};
      \draw[black] (-3.5,-1) node[right] {$\longleftrightarrow$};
      \draw[black] (n0) -- (-6,0.5);
      \draw[black] (n0) -- (n1);
      \draw[black] (n0) -- (n2);
      \draw[black] (-6,-1.5) node[] {$\cdots$};
      \draw[black] (-7,-1.7) node[below] {$T_1$};
      \draw[black] (n0) node[left] {$T_0~$};
      \draw[black] (-5,-1.7) node[below] {$T_l$};
      \draw[black] (-8.2,-1) node[] {$T=$};
    \end{tikzpicture}\]
\begin{proof}
	Since 
	\beq\label{eq:decomp.Stab}\Stab(T)\cong \symS_{n_0}\times \prod_{j=1}^l\left(\Stab(T_{j})^{r_j}\rtimes \symS_{r_j}\right) \eeq
	by definition, \eqref{ex:chUT} is immediate. See also \cite[Lemmas~4.3]{CKL2}. 
	Now, \eqref{ex:omegachUT} follows from \eqref{ex:chUT} by \cite[I.8.Example 1]{Mac}.
\end{proof}

The following is immediate from Lemma~\ref{lem:decomp.wrt}.
\begin{lemma}\label{lem:decomp.chiUT}
	In the above notations, 
	\[\begin{split}
		&\stan_{U_T}(m)=\multiset{m}{n_0}\cdot \prod_{j=1}^l\multiset{\stan_{U_{T_j}}(m)}{r_j}, ~\text{ and}\\
		&\stanomega_{U_T}(m)=\binom{m}{n_0}\cdot \prod_{r_j:\;\mathrm{even}}\multiset{\stan_{U_{T_j}}(m)}{r_j}\cdot\prod_{r_j:\;\mathrm{odd}}\binom{\stan_{U_{T_j}}(m)}{r_j}
	\end{split}
	\]
	as polynomials in $m$. \hfill $\square$
\end{lemma}

Now one can interpret the symmetric function $\ch_{\symS_n}U_{T}$ as a symmetric function encoding colorings of the inputs of $T$ as follows. For a labeled weighted rooted tree $T$, let $I_T$ be the set of the (labeled) inputs of $T$ and let 
\[\sK_T:=\{\kappa : I_T\to \Z_{\ge 1}\}\]
be the set of \emph{colorings} of the inputs of $T$ (by the colors indexed by positive integers). 
We consider two colorings to be the same if they are in the same orbit of $\Stab(T)\subset \symS_{I_T}$. 
Given a coloring  $\kappa$ of $T$, we denote 
\[x^\kappa := \prod_{i \in I_T} x_{\kappa(i)}.\]

\begin{proposition}\label{prop:coloring.UT}
	Let $T\in \sT_{n,k}^{\mathrm{lab}}$. Then
    \beq\label{eq:coloring.UT}\ch_{\symS_n}{U_T}=\sum_{\kappa\in \sK_T}x^\kappa. \eeq
\end{proposition}

\begin{proof}
  This is a consequence of P\'olya theory. See Theorem 7.24.4 and (7.120) in \cite{Stabook}.
  However, for the reader's convenience, we provide a proof using Lemma~\ref{lem:decomp.wrt} and the induction on the number of vertices of $T$.
  
  When $T$ has only one vertex, the root, then $\ch_{\symS_n}U_T=\sfh_n$ and \eqref{eq:coloring.UT} is immediate. 
  Suppose that $\lvert V(T)\rvert >1$.
The identity 
 \[ \sfh_{r}\circ \sum_{\kappa\in \sK_{T}}x^\kappa = \sum_{\kappa_1, \kappa_2, \ldots, \kappa_{r}\in \sK_{T}}x^{\kappa_1}x^{\kappa_2}\cdots x^{\kappa_{r}}\]
is immediate from the definition of the plethysm product (cf.~\cite[I.8]{Mac}). 
 Combining this with the decomposition in Lemma~\ref{lem:decomp.wrt}, we have
  \[
  \ch_{\symS_n}U_T=\left(\sum_{\kappa_0\in \sK_{T_0}}x^{\kappa_0}\right)\prod_{j=1}^l\sum_{\kappa_{j1},\kappa_{j2},\cdots,\kappa_{jr_j}\in \sK_{T_j}}x^{\kappa_{j1}}x^{\kappa_{j2}}\cdots x^{\kappa_{jr_j}}=\sum_{\kappa\in \sK_T}x^\kappa 
  \]
  where the first equality holds by the induction hypothesis applied to  each $T_j$ and the second equality holds since the product maps \[\kappa_0\times \prod_{1\leq j\leq l,1\leq i\leq r_j}\kappa_{ji}\ :\ I_{T_0}\times \prod_{1\leq j\leq l, 1\leq i\leq r_j} I_{T_j}\to \Z_{\ge 1}\] are precisely the colorings of $T$.
\end{proof}

An analogous statement holds for $\omega(\ch_{\symS_n}T)$ if we 
imposes a ``properness'' condition on colorings of the inputs as follows.  
For each vertex $v\in V(T)$, there exists a unique weighted rooted subtree $T_v$ of $T$ that satisfies all the conditions (1)--(4) in Lemma~\ref{lem:decomp.wrt} with $v$ in the place of $v_i$. In particular, $T_v$ has $v$ as its root vertex. 
\begin{definition}\label{def:coloring.proper.UT}
  Let $T\in \sT_{n,k}^{\mathrm{lab}}$. We call a coloring of the inputs of $T$ \emph{proper} if it satisfies the following two conditions: 
	\begin{enumerate}
		\item if two different inputs are attached to the same vertex, then they are colored by different colors; 
		\item if two disjoint weighted rooted subtrees of the form $T_v$ are isomorphic by the action of $\Stab(T)$ and each subtree has an \emph{odd} number of inputs, then they are non-isomorphic as colored weighted rooted trees.
	\end{enumerate}
\end{definition}

We denote by $\sK'_T$ 
 the set of all proper colorings of the inputs of $T$. 

\begin{proposition}\label{prop:coloring.proper.UT}
	Let $T\in \sT_{n,k}^{\mathrm{lab}}$. Then,
    \beq\label{eq:coloring.proper.UT}\omega (\ch_{\symS_n}{U_T})=\sum_{\kappa\in \sK'_T}x^\kappa. \eeq
\end{proposition}

\begin{proof}
	The proof is essentially the same as that of Proposition~\ref{prop:coloring.UT}, except that we need an additional identity 
	 \[ \sfe_{r}\circ \sum_{\kappa\in \sK'_{T}}x^\kappa = \sum_{\substack{\kappa_1, \kappa_2, \ldots, \kappa_{r}\in \sK'_{T}\\ \text{distinct}}}x^{\kappa_1}x^{\kappa_2}\cdots x^{\kappa_{r}}\]
	 for any $r\geq1$ and weighted rooted trees $T$, which is also immediate from the definition of the plethysm product (\cite[I.8]{Mac}). We omit the details.
\end{proof}
We will call the right hand side of \eqref{eq:coloring.proper.UT} \emph{the chromatic symmetric function} of a labeled weighted rooted tree $T$, which is clearly a symmetric function by definition. 
Proposition~\ref{prop:coloring.proper.UT} shows that the Frobenius characteristic of the $\omega$-dual of the representation $U_T$ is an analogue for $T$ of the chromatic symmetric function of a graph. However, in our case, the coloring is on inputs not on vertices, the properness condition is different from the usual one, and two colorings in the same orbit of $\Stab(T)$ are identified. 
 
\smallskip
The following is now immediate.  
\begin{corollary} \label{cor:coloring.chiUT}
Let $T\in \sT_{n,k}^{\mathrm{lab}}$ and $m_0\geq1$. Then
\begin{enumerate}
  \item $\stan_{U_T}(m_0)$ is the number of colorings of the inputs of $T$ by $m_0$ colors.
  \item $\stanomega_{U_T}(m_0)$ is the number of proper colorings of the inputs of $T$ by $m_0$ colors. 
\end{enumerate}
Moreover, $\stan_{U_T}(m)=(-1)^n\stanomega_{U_T}(-m)$ and it has nonnegative coefficients. \hfill $\square$
\end{corollary}
\begin{proof}
	(1) and (2) are immediate from Propositions~\ref{prop:coloring.UT} and \ref{prop:coloring.proper.UT} respectively. 
	The last assertion follows from Lemma~\ref{l:omega}, \eqref{eq:chi.U_H} and the fact that $U_T$ is a permutation representation.
\end{proof}
\begin{remark}
	Although Corollary~\ref{cor:coloring.chiUT}(1) can be viewed as a special case of Example~\ref{ex:chi.U_H} with $H=\Stab(T)$, we present an independent proof above 
	using 
	Lemma~\ref{lem:decomp.wrt}, to provide a unified proof of (1) and (2) in Corollary~\ref{cor:coloring.chiUT}.
\end{remark}

\begin{example}\label{ex:coloring1}
	Suppose that $T$ has only one vertex, namely the root, and all the $n$ inputs are attached to the root, as in Example~\ref{ex:stab1}. Then, 
	\[\ch_{\symS_n}U_T=\sfh_n\quad \text{ so that }\quad \stan_{U_T}(m)=\multiset{m}{n}.\] 
	This counts the number of colorings of the $n$ inputs of $T$ by $m$ colors. 
\end{example}

\begin{example}\label{ex:coloring2}
	Suppose that $T$ has only two vertices with $a$ inputs attached to the root and the remaining $n-a$ inputs attached to the other non-root vertex, as in Example~\ref{ex:stab2}. Then,
	\[\ch_{\symS_n}U_T=\sfh_a \sfh_{n-a} \quad \text{ so that }\quad \stan_{U_T}(m)=\multiset{m}{a} \multisetBig{m}{n-a}\]  
	This counts the number of colorings of the inputs of $T$ by $m$ colors. 
\end{example}
\begin{example}\label{ex:coloring3}
	Suppose that $T$ be the weighted rooted tree with three vertices, given  in Example~\ref{ex:stab3}. 
	Then, $\ch_{\symS_n}(U_T)=\sfh_{n-2a}\cdot (\sfh_2\circ \sfh_a)$ so that
	\[
	\stan_{U_T}(m)=\multiset{m}{n-2a} \multiset{\multiset{m}{a}}{2}.\] 
	This counts the number of colorings of the inputs of $T$ by $m$ colors. 
\end{example}
\begin{remark}[$\omega$-twisted version]\label{rem:coloring.omega}
	In the above examples, one can easily see that taking $\omega$ corresponds to restricting to proper colorings.
	Indeed, 
	\[\stanomega_{\sfh_n}(m)=\binom{m}{n}\]
	counts the number of colorings of $n$ inputs not allowing repetition of colors. Similarly, from the formula of $\stan_{U_T}(m)$ in Example~\ref{ex:coloring3}, 
	it follows that
	\[\begin{split}
		\stanomega_{U_T}(m)=(-1)^n&\stan_{U_T}(-m)
		=\begin{cases}
			\displaystyle{\binom{m}{n-2a}\cdot \multiset{\binom{m}{a}}{2}} &\text{ if $a$ is even,} \\
			\displaystyle{\binom{m}{n-2a}\cdot \binom{\binom{m}{a}}{2}} &\text{ if $a$ is odd.}
		\end{cases}
	\end{split}
	\]
	This counts the number of proper colorings of the inputs of $T$ by $m$ colors.
\end{remark}

\bigskip

\section{Asymptotic behavior and log-concavity}\label{s:asymp}
In this section, we study the asymptotic behavior and log-concavity of the polynomials $\charQ_n$ and $\charP_n$ as $n$ grows. 
Since they are polynomials in two variables $m$ and $t$, we may consider several notions of log-concavity in Definition~\ref{def:lc}. 

We write 
$$\charP_{n,k}(m)=\sum_{k\ge 0}\charP^j_{n,k}m^j,\quad  \charQ_{n,k}(m)=\sum_{k\ge 0}\charQ^j_{n,k}m^j$$
and provide asymptotic formulas for \begin{itemize}
\item (Theorem~\ref{thm:asymp.value}) $\charP_{n,k}(m_0)$ and $\charQ_{n,k}(m_0)$ with $m_0\ge 1$;
\item (Theorem~\ref{thm:mu.coeff}) $\charP^{n-j}_{n,k}$ and $\charQ^{n-j}_{n,k}$
\end{itemize}
as $n$ grows with $k,j$ fixed, based on the combinatorial description in the previous section.

We then establish the asymptotic degree log-concavity at $m=m_0$ (Corollary~\ref{cor:asymp.lc.value}), the asymptotic degree ultra log-concavity in length $n-j$ (Corollary~\ref{cor:lc.coeff.t}) and the asymptotic length log-concavity in degree $k$ (Corollary~\ref{cor:lc.coeff.m}) for fixed $j,k$, by
using the asymptotic formulas. 
The asymptotic formulas and the log-concavity results here generalize those in \cite{ACM} and \cite{CKL2}.

In the remainder, we will use  
the notation \[c_k=\frac{(k+1)^{k-1}}{k!} \and  d_k=\frac{(k+1)^{k-2}}{k!}.\]
These are the Cayley number $(k+1)^{k-1}$ of trees on $k+1$ labeled vertices, divided by $k!$ and $(k+1)!$ respectively.

\subsection{Asymptotic formulas}
Let us first find asymptotic formulas for 
$\charP_{n,k}(m_0)$ and $\charQ_{n,k}(m_0)$. 
\begin{theorem}\label{thm:asymp.value}
	Fix $m_0\geq 1$ and $k\geq0$. Then, as $n$ grows, 
	\[\begin{split}
		&\charQ_{n,k}(m_0)=\frac{c_k}{((k+1)m_0-1)!}n^{(k+1)m_0-1}+O(n^{(k+1)m_0-2}),\\
		&\charP_{n,k}(m_0)=\frac{d_k}{((k+1)m_0-1)!}n^{(k+1)m_0-1}+o(n^{(k+1)m_0-1}).
	\end{split} \]
\end{theorem}

\begin{remark}\label{re:19}
By Lemma~\ref{l:mult} (2), 
\[\charP_{n,k}(1)=\dim\, H^{2k}(\Mbar_{0,n})^{\symS_n} \and \charQ_{n,k}(1)=\dim\, H^{2k}(\Mbar_{0,n+1})^{\symS_n}.\]
Hence if we let $m_0=1$, 
Theorem~\ref{thm:asymp.value} recovers the asymptotic formula \cite[Theorem~1.9]{CKL2}. Likewise, Corollary~\ref{cor:asymp.lc.value} below with $m_0=1$ recovers \cite[Corollary 6.10]{CKL2}. 
\end{remark}

See \S\ref{S7.1} for a proof of Theorem~\ref{thm:asymp.value}. 
Here, we derive the formulas for $\charQ_{n,k}(m_0)$ when $k\leq 2$ as illustrative examples. 

\begin{example}[$k=0$] There exists only one tree in $\sT_{n,0}$, which is the one described in Example~\ref{ex:coloring1} with weight 0 at the root. Thus,
	\[\charQ_{n,0}(m_0)=\multiset{m_0}{n}=\binom{m_0+n-1}{m_0-1}=\frac{n^{m_0-1}}{(m_0-1)!}+O(n^{m_0-2}).\]
\end{example}
\begin{example}[$k=1$] We have two types of weighted rooted trees in $\sT_{n,1}$, described in Examples~\ref{ex:coloring1} and \ref{ex:coloring2} respectively. 
Hence, 
	\[\charQ_{n,1}(m_0)=\multiset{m_0}{n}+\sum_{a=1}^{n-3}\multiset{m_0}{a}\cdot \multisetBig{m_0}{n-a}.\]
	To obtain the asymptotic formula, notice that for $m_0\geq 1$,
	\[\begin{split}
		\sum_{a=0}^{n} \multisetBig{m_0}{a}\cdot \multisetBig{m_0}{n-a}=\multisetBig{2m_0}{n}
		&=\frac{n^{2m_0-1}}{(2m_0-1)!}+O(n^{2m_0-2}),
	\end{split}
	\]
	while $\multisetbody{m_0}{n}=\frac{n^{m_0-1}}{(m_0-1)!}+O(n^{m_0-2})$.
	Since $2m_0-1>m_0-1$ for $m_0\geq1$, 
	\[\charQ_{n,1}(m_0)=\frac{n^{2m_0-1}}{(2m_0-1)!}+O(n^{2m_0-2}).\]
\end{example}
\begin{example}[$k=2$]
	The contribution of weighted rooted trees $T$ in $\sT_{n,2}$ with at most two vertices, to $\charQ_{n,2}(m_0)$, is 
	\[\begin{split}
		&\multisetBig{m_0}{n}+\sum_{a=4}^{n-1}\multisetBig{m_0}{a}\cdot \multisetBig{m_0}{n-a}+\sum_{a=3}^{n-2}\multisetBig{m_0}{a}\cdot \multisetBig{m_0}{n-a}\\
		&\leq \multisetBig{m_0}{n}+2\sum_{a=0}^{n}\multisetBig{m_0}{a}\cdot \multisetBig{m_0}{n-a}=\frac{2n^{2m_0-1}}{(2m_0-1)!}+O(n^{2m_0-2}).
	\end{split}
	\]
	Now we compute the contribution of $T\in \sT_{n,2}$ with three vertices. There are three types depending on the shape of trees without inputs and $\Stab(T)$ so that 
	\[\begin{split}
		&\sum_{a\geq3,\; b\geq 2}\multisetBig{m_0}{a}\cdot\multisetBig{m_0}{b}\cdot \multisetBig{m_0}{n-a-b}\\
		&+\sum_{\substack{3\leq a<b\leq n,\\ a+b\leq n}}\multisetBig{m_0}{a}\cdot\multisetBig{m_0}{b}\cdot \multisetBig{m_0}{n-a-b}+\sum_{3\leq a<\frac{n}{2}}\multiset{\multisetBig{m_0}{a}}{2}\cdot\multisetBig{m_0}{n-2a}
	\end{split}\]
	where the last term is computed in Example~\ref{ex:coloring3}, and it is dominated by the second term.
	Hence, we find that the sum is equal to 
	\[(1+\frac{1}{2})\cdot\sum_{\substack{a,b,c\geq0,\\a+b+c=n}} \multisetBig{m_0}{a}\cdot \multisetBig{m_0}{b}\cdot\multisetBig{m_0}{c}=\frac{3}{2}\cdot \multisetBig{3m_0}{n}=\frac{3}{2}\cdot\frac{n^{3m_0-1}}{(3m_0-1)!}\] 
	modulo $O(n^{3m_0-2})$, as a function of  $n$.
	In particular,
	\[\charQ_{n,2}(m_0)=\frac{3}{2}\cdot \frac{n^{3m_0-1}}{(3m_0-1)!}+O(n^{3m_0-2}).\]
\end{example}

\bigskip

We next present asymptotic formulas for the coefficients $\charP^{n-j}_{n,k}$ and $\charQ^{n-j}_{n,k}$ of $\charP_{n,k}(m)$ and $\charQ_{n,k}(m)$ as $n$ grows.
Recall that the \emph{signless Stirling number of the first kind}, denoted by $c(n,j)$, is the number of elements $\sigma\in \symS_n$ with $j$ disjoint cycles.

\begin{theorem}\label{thm:mu.coeff}
	Let $k,j\in \ZZ_{\ge 0}$. Then, as a function of $n$, 
	\[ \charQ_{n,k}^{n-j}=c_k\cdot\frac{(k+1)^{n-j}c(n,n-j)}{n!}+O\left(\frac{k^nn^{2j}}{n!}\right)\ \ \ \text{and}\]
    \[ \charP_{n,k}^{n-j}=d_k\cdot\frac{(k+1)^{n-j}c(n,n-j)}{n!}+o\left(\frac{(k+1)^nn^{2j}}{n!}\right).\]
\end{theorem}

We remark that $O(\frac{(k+1)^{n-j}c(n,n-j)}{n!})=O(\frac{(k+1)^nn^{2j}}{n!})$, as $c(n,n-j)$ and $\frac{n^{2j}}{2^jj!}$ are asymptotically equal, that is, their ratio converges to 1 as $n\to \infty$. 

See \S\ref{S7.4} for a proof of Theorem~\ref{thm:mu.coeff}. 
\begin{remark}[Leading coefficients] \label{rem:leading.coeff}
	The first formula in Theorem~\ref{thm:mu.coeff} with $j=0$ and Lemma~\ref{l:mult} (1) imply that
	$$  \frac{h^{2k}(\Mbar_{0,n+1})}{n!}=  \charQ_{n,k}^n=c_k\cdot\frac{(k+1)^n}{n!}+O(\frac{k^n}{n!}).$$
	We thus obtain 
	\[\lim_{n\to \infty}\frac{h^{2k}(\Mbar_{0,n+1})}{(k+1)^n}= c_k\]
which is a result of Aluffi-Chen-Marcolli (cf.~\cite[Theorem~1.3]{ACM} with $n$ replaced by $n+1$).
\end{remark}
\begin{remark}[Linear terms]
	An analogous statement for $\charQ^j_{n,k}$ is not true. This failure can be observed even for the linear terms.
	Note that $\charQ_{n,k}$ is always divisible by $m$. One can see from \eqref{eq:chi.U_H} that $\charQ_{n,k}^1  = \frac{1}{n}$ whenever $n$ is a prime number. Indeed, when $n$ is prime, $\Stab(T)\subset \symS_n$ contains an $n$-cycle only if $T$ has only one vertex and $\Stab(T)=\symS_n$. Since the number of $n$-cycles in $\Stab(T)=\symS_n$ is precisely $(n-1)!$, the coefficient of $m$ is $\frac{(n-1)!}{n!}=\frac{1}{n}$.

    On the other hand, when $n=2p\geq 6$ for a prime $p$, the weighted rooted tree $T$ consisting of two non-root vertices attached to the root, each with $p$ inputs, contributes an additional $\frac{1}{n}$, provided that $T$ satisfies the valency conditions (cf.~Definition~\ref{y20}).
	
   More generally, by the same argument using \eqref{eq:chi.U_H}, one can check that 
	\beq\label{eq:qnk1}
    \charQ_{n,k}^1=\frac{1}{n}+\frac{1}{n}\sum_{k_0=0}^{k-2}N_{n,k}(k_0),
    \eeq
	where for $0\leq k_0\leq k-2$, the integers $N_{n,k}(k_0)$ are defined as
	\[N_{n,k}(k_0):=\left\lvert\left\{r\in \Z_{>1}:k_0+2\leq r\leq \frac{n-k+k_0}{2} ~\text{ and }~ r\mid gcd(n,k-k_0)\right\}\right\rvert.\]
	The number $N_{n,k}(k_0)$ is the contribution from the weighted rooted trees having $r$ non-root vertices attached to the root, each with $\frac{n}{r}$ inputs, where the root vertex has weight $k_0$ and non-root vertices have weight $\frac{k-k_0}{r}$. 
\end{remark}

\medskip

\subsection{Asymptotic log-concavity}

From Theorem~\ref{thm:asymp.value}, we obtain the asymptotic degree log-concavity at $m=m_0\ge 1$ of $\charP_n(m,t)$ and $\charQ_n(m,t)$. 
\begin{corollary}\label{cor:asymp.lc.value}
	Fix $m_0,k\geq 1$. Then we have 
	\[\begin{split}
		&\charP_{n,k}(m_0)^2\geq \charP_{n,k-1}(m_0)\charP_{n,k+1}(m_0) ~\text{ and}\\
		&\charQ_{n,k}(m_0)^2\geq \charQ_{n,k-1}(m_0)\charQ_{n,k+1}(m_0)
	\end{split}
	\]
	for sufficiently large $n$. Moreover, when $m_0>1$, the inequalities are strict.
\end{corollary}
\begin{proof}
	By Theorem~\ref{thm:asymp.value}, it suffices to show that the sequences 
	\[\left\{\frac{c_k}{((k+1)m_0-1)!}\right\}_{k\ge 0}  \and \left\{\frac{d_k}{((k+1)m_0-1)!}\right\}_{k\ge 0}\] 
	are log-concave.
	Since the case of $m_0=1$ was already proved in \cite{CKL2}, we may assume $m_0>1$.
	Then, one can check that
	\[\begin{split}
		\frac{(km_0-1)!((k+2)m_0-1)!}{(((k+1)m_0-1)!)^2}&= \frac{(km_0+m_0)}{km_0}\frac{(km_0+m_0+1)}{km_0+1}\cdots\frac{(km_0+2m_0-1)}{(km_0+m_0-1)}\\
		&=\prod_{i=0}^{m_0-1}\left(1+\frac{m_0}{km_0+i}\right)> 1+\frac{1}{k}=\frac{k+1}{k},
	\end{split}
	\]
	while $\frac{c_k^2}{c_{k-1}c_{k+1}}=\frac{(k+1)^{2k-1}}{k^{k-1}(k+2)^k}$. Combining these, we have
	\[\frac{c_k^2}{c_{k-1}c_{k+1}}\cdot \frac{(km_0-1)!((k+2)m_0-1)!}
	{(((k+1)m_0-1)!)^2}> 
	\left(1+\frac{1}{k^2+2k}\right)^k>1.\]
	This shows the second (strict) inequality in the assertion. Similarly, 
	\[\frac{d_k^2}{d_{k-1}d_{k+1}}\cdot \frac{(km_0-1)!((k+2)m_0-1)!}
	{(((k+1)m-1)!)^2}> 
	\left(1+\frac{1}{k^2+2k}\right)^{k-1}\geq1.\]
	This shows the first (strict) inequality in the assertion.
\end{proof}

\medskip

From Theorem~\ref{thm:mu.coeff}, we establish two asymptotic log-concavities of $\charP_n(m,t)$ and $\charQ_n(m,t)$.   
First, we show the asymptotic degree \emph{ultra} log-concavity in length $n-j$ as follows. 
\begin{corollary}\label{cor:lc.coeff.t}
	Fix $j\geq 0$ and $k\geq 1$. Then for sufficiently large $n$, 
	\[\begin{split}
		&\left(\frac{
        \charP_{n,k}^{n-j}}{\binom{n-3}{k}}\right)^2\geq \left(\frac{
        \charP_{n,k-1}^{n-j}}{\binom{n-3}{k-1}}\right)\left(\frac{
        \charP_{n,k+1}^{n-j}}{\binom{n-3}{k+1}}\right) ~\text{ and}\\
		&\left(\frac{
        \charQ_{n,k}^{n-j}}{\binom{n-2}{k}}\right)^2\geq \left(\frac{
        \charQ_{n,k-1}^{n-j}}{\binom{n-2}{k-1}}\right)\left(\frac{
        \charQ_{n,k+1}^{n-j}}{\binom{n-2}{k+1}}\right).
	\end{split}
	\]
\end{corollary}
\begin{proof}
	When $n$ is sufficiently large, these are by Theorem~\ref{thm:mu.coeff} equivalent to log-concavity of $\left\{d_k(k+1)^n/\binom{n-3}{k}\right\}_k$ and $\left\{c_k(k+1)^n/\binom{n-2}{k}\right\}_k$ at given $k$ respectively. One can easily check that the latter are true for all $k\geq 1$.
\end{proof}
This result generalizes the asymptotic ultra log-concavity of the Poincar\'e polynomial of $\Mbar_{0,n}$, previously discovered in \cite{ACM}, which corresponds to the case of $j=0$.

Next, we examine the length log-concavity of the characteristic polynomials $\charP_{n}(m,t)$ and $\charQ_{n}(m,t)$ in degree $k$. 
\begin{corollary}\label{cor:lc.coeff.m}
	Fix $j\geq 1$ and $k\geq 0$. Then for sufficiently large $n$,
	\[\begin{split}
		&\left(
        \charP_{n,k}^{n-j}\right)^2\geq \left(
        \charP_{n,k}^{n-j-1}\right)\cdot\left(
        \charP_{n,k}^{n-j+1}\right) ~\text{ and}\\
		&\left(
        \charQ_{n,k}^{n-j}\right)^2\geq \left(
        \charQ_{n,k}^{n-j-1}\right)\cdot\left(
        \charQ_{n,k}^{n-j+1}\right).
	\end{split}
	\]
\end{corollary}
\begin{proof}
By Theorem~\ref{thm:mu.coeff}, when $n$ is sufficiently large, this reduces to the log-concavity of $\{c(n,n-j)\}_j$, which follows immediately from the identity $m(m+1)\cdots(m+n-1)=\sum_{j=0}^n c(n,j)m^j$ and Lemma~\ref{lem:realrooted}.
\end{proof}

\bigskip
\section{Proofs of asymptotic formulas}\label{s:proofs}
In this section, we provide proofs of Theorem~\ref{thm:asymp.value} and Theorem~\ref{thm:mu.coeff}. The proofs are based on the combinatorial description involving the weighted rooted trees. A central observation is that the dominant terms in both asymptotic formulas arise from the rooted tree with the maximum possible number of vertices, which accounts for the appearance of the Cayley numbers in the formula.

\subsection{Proof of Theorem~\ref{thm:asymp.value}} \label{S7.1}
Let $\overline \sT_k$ be the set of isomorphism classes of weighted rooted trees with total weight $k$ and no inputs. Let 
\[F_n:\sT_{n,k}\lra \overline \sT_k\] 
be the forgetful map which forgets all the inputs of a given $T\in \sT_{n,k}$. 
In other words, the fiber of $\overline T\in \overline\sT_k$ consists of all possible ways of attaching $n$ inputs to the vertices of $\overline T$, subject to the valency conditions (cf.~Definition~\ref{y20}).

Decompose $\charQ_{n,k}$ with respect to $F_n$: 
\[\charQ_{n,k}=\sum_{\overline T\in \overline \sT_k}\charQ_{n,\overline T}, \quad \text{ where }\charQ_{n,\overline T}:=\sum_{T\in F_n^{-1}(\overline T)}\stan_{U_T}.\]
Suppose that $j:=\lvert V(\overline T)\rvert \leq k$. Then, by Corollary~\ref{cor:coloring.chiUT}, $\charQ_{n,\overline T}(m_0)$ is 
bounded by the number of all possible ways of attaching $n$ inputs to the vertices of $\overline{T}$ and coloring them using $m_0$ colors, which is given by 
\[\sum_{\substack{n_1+\cdots+n_j=n,\\ n_i\geq0 }}\prod_{i=1}^j\multiset{m_0}{n_i}=\multiset{jm_0}{n}=\frac{n^{jm_0-1}}{(jm_0-1)!}+O(n^{jm_0-2}).\]
In particular, 
$\charQ_{n,\overline T}(m)=O(n^{(k+1)m_0-2})$ for all $\overline T$ with at most $k$ vertices. 

Suppose that $\lvert V(\overline T)\rvert=k+1$. By a similar argument, we have 
\[\charQ_{n,\overline T}(m_0)=\frac{1}{\lvert \Aut(\overline T)\rvert}\cdot \frac{n^{(k+1)m_0-1}}{((k+1)m_0-1)!}+O(n^{(k+1)m_0-2}).\]
Furthermore, since $k+1$ is the maximal possible number of vertices that $\overline T\in \overline\sT_k$ can have, every non-root vertex of $\overline T$ has weight one while the root vertex has weight zero. In particular, automorphisms of $\overline T$ as weighted rooted trees are precisely those as rooted trees. Therefore, by the same argument as in the proof of \cite[Theorem 6.3]{CKL2}, we have
\beq\label{eq:ck}\sum_{\overline T\in \overline \sT_k,~\lvert V(\overline T)\rvert=k+1}\frac{1}{\lvert\Aut(\overline T)\rvert}=\frac{(k+1)^{k-1}}{\lvert \symS_k\rvert}=c_k\eeq
where $\overline T$ are considered as trees with $k+1$ vertices, and $(k+1)^{k-1}$ is the number of trees on $k+1$ labeled vertices by Cayley's tree formula.  
Consequently,
\beq\label{eq:mu_nk.BigO}\charQ_{n,k}(m_0)=\sum_{\overline T\in \overline T_k}\charQ_{n,\overline T}(m_0)=c_k\cdot \frac{n^{(k+1)m_0-1}}{((k+1)m_0-1)!}+O(n^{(k+1)m_0-2}),\eeq
which implies the first formula.  

For the second formula, divide both sides of \eqref{eq:walcrossing.charpoly} by $n^{(k+1)m_0-1}$ and take their limits as $n\to \infty$.  
Then,
	\[\begin{split}
		&\lim_{n\to\infty }\frac{\charP_{n,k}(m_0)+\charP_{n,k-1}(m_0)}{n^{(k+1)m_0-1}}=\frac{c_k}{((k+1)m_0-1)!}\\
		&-\frac{1}{2}\sum_{j=0}^{k-1}\frac{c_jc_{k-1-j}}{((j+1)m_0-1)!((k-j)m_0-1)!}\cdot\lim_{n\to\infty }\sum_{h=2}^{n-2}\frac{h^{(j+1)m-1}(n-h)^{(k-j)m_0-1}}{n^{(k+1)m_0-1}},
	\end{split}
	\]
	due to \eqref{eq:mu_nk.BigO}. Indeed, for integers $a,b,c\geq0$,
	\[\lim_{n\to\infty}\sum_{h=2}^{n-2}h^a(n-h)^b\frac{1}{n^{c+1}}= \int_0^1x^a(1-x)^bdx=\frac{a!b!}{(a+b+1)!}\]
	if $a+b=c$, and hence it vanishes if $a+b<c$. This also implies that
	\[\begin{split}
		&\lim_{n\to\infty }\frac{\charP_{n,k}(m_0)+\charP_{n,k-1}(m_0)}{n^{(k+1)m_0-1}}=\frac{1}{((k+1)m_0-1)!}\Big(c_k- \frac{1}{2}\sum_{j=0}^{k-1}c_jc_{k-1-j}\Big),
	\end{split}
	 \] 
	which is equal to ${d_k}/{((k+1)m_0-1)!}$ as checked in \cite[Theorem~6.9]{CKL2}.
    Now the second formula in Theorem~\ref{thm:asymp.value} follows by an induction on $k$.
\hfill \mathqed
\\

The rest of this section is devoted to a proof of Theorem~\ref{thm:mu.coeff}.
It requires certain facts about $\Stab(T)$, as described in the following subsections.

\subsection{Stratifications of $\Stab(T)$} \label{S7.2}
For a labeled weighted rooted tree $T$, 
let $F(T)$ denote the weighted rooted tree without inputs obtained from $T$ by forgetting the inputs. 
Recall that $\Aut(F(T))\subset \symS_{V(F(T))}$ consists of elements whose actions on $V(T)=V(F(T))$ preserve the isomorphism class of $F(T)$ as a weighted rooted tree. Then, the natural group homomorphism $\Stab(T)\to \symS_{V(T)}$ factors through $\Aut(F(T))$.

Let $\lambda(T)$ denote the partition of the set of the inputs according to the vertices of $T$ they are attached to. This defines a subgroup $\symS_{\lambda(T)}\subset \Stab(T)$.

\begin{lemma}\label{l:Stab.decomp}
	Let $T$ be a labeled rooted tree. Let $\bar T= F(T)$. Then $\symS_{\lambda(T)}$ is the kernel of the group homomorphism $\Stab(T)\to \Aut(\bar T)$.
\end{lemma}
\begin{proof}
	This follows from the inductive construction of $T$ and $\Stab(T)$ in Lemma~\ref{lem:decomp.wrt} and \eqref{eq:decomp.Stab}, and that of $\bar T$ and $\Aut(\bar T)$.
\end{proof}

\begin{example}\label{ex:wrt.trileg}
	Let $\bar T$ be the weighted rooted tree without inputs and with $V(\bar T)=\{v_0,v_1,v_2,v_3\}$, where $v_0$ is the root with weight 0 and the other vertices are adjacent to $v_0$ and have weight 1. 
	
	We attach labeled inputs to vertices of $\bar T$ to obtain a labeled weighted rooted tree $T$. Assume that $I=\{1,\cdots,9\}$ is the set of all inputs and $I_{v_j}$ is the subset of the inputs attached to $v_j$ defined as follows:
	\[I=I_{v_0}\sqcup I_{v_1}\sqcup I_{v_2}\sqcup I_{v_3}\]
	with $I_{v_0}=\emptyset$, 
	$I_{v_1}=\{1,4,7\}$, $I_{v_2}=\{2,5,8\}$ and $I_{v_3}=\{3,6,9\}$.
	In particular, 
	\[\Stab(T)=(\symS_{I_{v_1}}\times \symS_{I_{v_2}}\times \symS_{I_{v_3}})\rtimes \symS_3~\subset~ \symS_{I}=\symS_9,\] 
	with a splitting given by   
	\[\symS_3\lra \Stab(T), \quad \tau ~\mapsto ~\big(3i+j~\mapsto~ 3i+\tau(j)\in I\big)_{0\leq i\leq 2, 1\leq j\leq 3}.\]
	The splitting  identifies $\Aut(\bar T)=\symS_3$, as a natural quotient of $\Stab(T)$.

    More concretely, an element of $\Stab(T)$ is the permutation $\sigma$ in $\symS_I$ that preserves the partition $I_{v_1}\sqcup I_{v_2}\sqcup I_{v_3}$, that is, $\sigma(I_{v_j}) = I_{\tau(v_j)}$ for each $j$ for some permutation $\tau\in \symS_{\{v_1,v_2,v_3 \}} $. In this case, $\tau$ is the image of $\sigma$ by the map $\Stab(T)\to \Aut(\bar T)$. 
\end{example}

\begin{definition}
	Let $T$ be a labeled weighted rooted tree. Let $\bar T=F(T)$. Let $n$ be the number of the inputs of $T$ so that $\Stab(T)\subset \symS_n$.
	\begin{enumerate}
		\item Define the \emph{length} of an element $\sigma$ in $\Stab(T)$, denoted by $\ell(\sigma)$, as the number of disjoint cycles in $\sigma$  as an element of $\symS_n$, or equivalently, the number of its orbits in the set of labels.  This is also the length of the partition of $n$ associated to $\sigma\in \symS_n$. 
		\item Let $\Stab^\ell(T)\subset \Stab(T)$ be the subset of elements of length $\ell>0$. 
		\item Let $\Stab_{\tau}(T)\subset \Stab(T)$ be the preimage of $\tau\in \Aut(\bar T)$ under the canonical map $\Stab(T)\to \Aut(\bar T)$. 
		\item Let $\Stab^\ell_\tau(T)=\Stab^\ell\cap \Stab_\tau(T)$ for each $\ell$ and $\tau$. 
	\end{enumerate}
	In particular, $\Stab(T)=\bigsqcup_{\ell>0,\tau\in \Aut(\bar T)}\Stab^\ell_\tau(T)$. 
	
	We define similarly for unlabeled $T$, in which case every object above is well defined only (as a subset of $\symS_n$) up to conjugation.
\end{definition}

\begin{example}\label{ex:tau}
	Let $T$, $\bar T$, $I$ and $I_{v_j}$ be as in Example~\ref{ex:wrt.trileg}. 
	Then the permutation $\sigma=(1,5,9,4,2,6,7,8,3)\in \symS_I=\symS_9$ lies in $\Stab(T)$.
    Moreover, its image under the map $\Stab(T)\to \Aut(\bar T)=\symS_3$ is $\tau:=(1,2,3)$, since $\sigma$ cyclically permutes the subsets $I_{v_1}\to I_{v_2}\to I_{v_3}\to I_{v_1}$. In particular, $\sigma\in \Stab_\tau(T)$.
\end{example}

\begin{example}\label{ex:chi.UT}
	Let $T$ be a labeled weighted rooted tree with $n$ inputs. Then, applying \eqref{eq:chi.U_H} in Example~\ref{ex:chi.U_H} to $H=\Stab(T)$, we obtain
	\[\stan_{U_T}(m)=\sum_{j=1}^n \frac{\lvert \Stab^j(T)\rvert}{\lvert \Stab(T)\rvert}m^j=\sum_{j=1}^n\sum_{\tau\in \Aut(\bar T)} \frac{\lvert \Stab^j_\tau(T)\rvert}{\lvert \Stab(T)\rvert}m^j.\]
	This is the same for unlabeled $T$.
\end{example}

\medskip

The first step in the proof of Theorem~\ref{thm:mu.coeff} is to decompose
the coefficient $\charQ^{n-j}_{n,k}$. 
To explain this, we first consider the decomposition of $\charQ_{n,k}$ with respect to the map $F_n:\sT_{n,k}\to \bar \sT_k$ as follows:
\[\charQ_{n,k}=\sum_{\bar T\in \bar \sT_k}\charQ_{n,\bar T}, \quad \text{ where } ~\charQ_{n, \bar T}:=\sum_{T\in F^{-1}_n(\overline T)} \stan_{U_T} \]
for $\bar T \in \overline \sT_k$.
By definition,  $F_n^{-1}(\overline T)\subset \sT_{n,k}$ consists of weighted rooted trees obtained by attaching $n$ inputs to the vertices of $\bar T$.

This decomposes further by Example~\ref{ex:chi.UT}: taking the coefficients of $m^{n-j}$,
	\beq \label{eq:coeff}
	\charQ_{n,\bar T}^{n-j}= 
    \sum_{\tau\in \Aut(\bar T)}\sum_{T\in F_n^{-1}(\bar T)}\frac{\lvert\Stab^{n-j}_\tau(T)\rvert}{\lvert\Stab(T)\rvert}.
    \eeq
	As we will see later, the summand in \eqref{eq:coeff} for $\tau =e$ dominates the others as $n$ grows. So we first obtain an asymptotic formula for the case of $\tau=e$.
	
	\medskip
    
When $\tau = e$, the numerator $\lvert\Stab_e^{n-j}(T)\rvert$ counts elements in $\Stab_e(T) \cong \symS_{\lambda(T)}$ with $n-j$ disjoint cycles, where the number of inputs attached to each vertex of $\bar{T}$ can vary. Since $\bar{T}$ is fixed, so is $\ell(\lambda(T))=\lvert V(\bar T)\rvert$.
Thus, the summand for $\tau =e$ is controlled by the $\lvert V(\bar{T})\rvert$-th power of the generating function of the signless Stirling numbers of the first kind defined as follows. 

Recall that $c(n,j)$ satisfy
\beq\label{eq:Stirling.poly}m(m+1)\cdots(m+n-1)=\sum_{j=0}^n c(n,j)m^j \and c(0,0)=1.\eeq
The (double) generating function of these numbers is known as follows:
\[\mathfrak{S}:=\sum_{n\geq1}\sum_{j=1}^n\frac{c(n,j)}{n!}u^nv^j=(1-u)^{-v}
\quad(=\exp(-v\log(1-u)))\]
where $u$ and $v$ are formal variables.

For notational convenience, we write  for given $j,k\geq 0$,
\[A_{j,k}(n)=\mathrm{coeff}_{m^{n-j}}\multiset{m(k+1)}{n}=\frac{(k+1)^{n-j}c(n,n-j)}{n!}.\]
Note that $\lim_{n\to \infty}c(n,n-j)/(\frac{n^{2j}}{2^jj!})=1$. In particular, $O(A_{j,k-1})=O(\frac{k^nn^{2j}}{n!})$.

\begin{proposition}\label{p:summand.identity}
	Fix $j$ and $k\geq0$. Then as $n$ grows, 
	\[\sum_{\bar T\in \bar \sT_k}\sum_{T\in F_n^{-1}(\bar T)}\frac{\lvert \Stab_e^{n-j}(T)\rvert}{\lvert \Stab(T)\rvert}=  c_k\cdot
	A_{j,k}(n) 
	+O\left(\frac{k^nn^{2j}}{n!}\right).\]
\end{proposition}
\begin{proof}
The proposition immediately follows from \eqref{eq:ck} and the following: for $\bar T\in \bar \sT$ with $N$ vertices,
	\beq\label{y21} \sum_{T\in F_n^{-1}(\bar T)}\frac{\lvert \Stab_e^{n-j}(T)\rvert}{\lvert \Stab(T)\rvert}=  \frac{1}{\lvert \Aut(\bar T)\rvert}\cdot A_{j,N-1}(n) 
	+O\left(\frac{(N-1)^nn^{2j}}{n!}\right).\eeq

	The left hand side of \eqref{y21} is the summation over all possible $N$-tuples of partitions of $n_1,\cdots, n_N$ with $\sum n_j=n$ subject to the valency conditions, where each summand is multiplied by $\lvert \Aut(\bar T)\rvert$. Since the valency conditions are given in the form $n_j\geq m_j$ for some $m_j$ depending of the shape of $\bar T$ and the weights on the vertices of $\bar T$, the left hand side is asymptotically equal to the summations over such $N$-tuples with any $n_j\geq 0$, upon multiplication by $\lvert \Aut(\bar T)\rvert$. Therefore, 
	\[\sum_{T\in F_n^{-1}(\bar T)}\frac{\lvert \Stab_e^{n-j}(T)\rvert}{\lvert \Stab(T)\rvert}=\frac{1}{\lvert \Aut(\bar T)\rvert} \sum_{\substack{n_1+\cdots+n_N= n,\\ j_1+\cdots+j_N= j,\\ n_i,j_i\geq0 \text{ for all }i}}\prod_{i=1}^N\frac{c(n_i,n_i-j_i)}{n_i!}\]
	modulo $O(\frac{(N-1)^nn^{2j}}{n!})$. The summation on the right hand side of \eqref{y21} is equal to the coefficient of $u^nv^{n-j}$ in
	\[\mathfrak{S}^N=(1-u)^{-Nv}=\exp(-Nv\log(1-u))\]
	when expanded to a power series.  
	Since 
	$$\frac{d^n}{d u^n}(1-u)^{-Nv}=\prod_{i=0}^{n-1}(Nv+i)\cdot (1-u)^{-Nv-n},$$
	the coefficient of $u^n$ in $\mathfrak{S}^N$ is
	$$\frac{1}{n!}\frac{d^n}{d u^n}(1-u)^{-Nv}|_{u=0}=\frac{1}{n!}\prod_{i=0}^{n-1}(Nv+i).$$
	Hence, by \eqref{eq:Stirling.poly}, the coefficient of $u^nv^{n-j}$ in $\mathfrak{S}^N$ is equal to $c(n,n-j)N^{n-j}$ divided by $n!$. 
\end{proof}

\medskip

It remains to show that the summands in \eqref{eq:coeff} corresponding to $\tau\neq e$ are asymptotically dominated by the summand corresponding to $\tau=e$. To this end, we introduce the notion of quotient trees by $\tau$ in the following subsection. 
We then compare the summands corresponding to $\tau\neq e$ with the summations associated to the quotient trees by $\tau$. Roughly speaking, the fact that these quotient trees are strictly smaller than the original tree ensures that the associated summations are asymptotically dominated by the one corresponding to $\tau= e$.   

\subsection{Quotient trees} \label{S7.3}
	Let $\bar T$ be a weighted rooted tree without inputs, where its vertices are ordered. Let $\tau\in \Aut(\bar T)\subset \symS_{V(\bar T)}$ be an automorphism of $\bar T$. Then there is a natural quotient $\bar T/\tau$ of $\bar T$, regarded as an element of $ \bar \sT$, 
	by the $\tau$-actions on $V(\bar T)$ and $E(\bar T)$ so that
	\[V(\bar T/\tau)=V(\bar T)/\tau \and E(\bar T/\tau)=E(\bar T)/\tau\] respectively. 
	The weight function $w:V(\bar T)\to \Z_{\geq0}$ of $\bar T$ naturally factors through $V(\bar T/\tau)\to \Z_{\geq0}$. We assign this as the weight function of $\bar T/\tau$. 

\medskip

Now let $T$ be a labeled weighted rooted tree.
Let $\bar T=F(T)$. Suppose that $\tau$ is in the image of $\Stab(T)\to \Aut(\bar T)$. Let $n_v$ denote the number of inputs attached to each vertex $v$ of $T$. Then we have $n_u=n_v$ whenever two vertices $u$ and $v$ are in the same $\tau$-orbit. Indeed, any $\sigma\in \Stab_\tau(T)$ induces an isomorphism between two rooted subtrees $T_u$ and $T_v$ in $T$ that have $u$ and $v$ as their root vertices, respectively.

\begin{definition}
	We define $T/\tau$
	to be the weighted rooted tree obtained from $\bar T/\tau$ by attaching $n_v$ inputs to the image of each $v$ in $ V(\bar T/\tau)=V(T)/\tau$.
\end{definition}
Roughly speaking, $T/\tau$ is obtained from $T$ by identifying all the weighted rooted subtrees $T_v$ for $v$ having the same images in $V(T)/\tau$.

\medskip

It is natural to expect that $\sigma\in \Stab_\tau(T)$ induces an element in $\Stab_e(T/\tau)$, since its appropriate power (depending on $v$) will stabilize each vertex $v$ and permute the $n_v$ inputs attached to $v$. 
In what follows, we construct a map 
\[\Stab_\tau(T)\lra \Stab_e(T/\tau)\]
and show that it has some useful properties. 
 
Let $I$ be the set of labels on the inputs of $T$, and for each $v\in V(T)$, let $I_v$ be the set of the labels on the inputs attached to $v$. Hence, $I:=\bigsqcup_{v\in V(T)}I_v$,
and $\Stab_\tau(T)$ is a subset of the symmetric group $\symS_I$ on $I$.   

Let $d_v>0$ be the minimal positive integer such that $\tau^{d_v}(v)=v$ for each $v\in V(T)$. Then there exists a map 
 \[\varphi_v:\Stab_\tau(T)\lra  \symS_{I_v}\subset \Stab(T_v)\]
 that sends $\sigma$ to $\sigma^{d_v}$ restricted to the set $I_v$ of the inputs attached to $v$. 
As $\Stab_e(T/\tau)\cong \symS_{\lambda(T/\tau)}$,
the desired map can be obtained as the product of the maps $\varphi_v$ where $v$ runs over representatives of $V(T)/\tau$. 

Choose representatives $v_1,\cdots,v_p\in V(T)$, one for each $\tau$-orbit, so that $V(T)/\tau\cong \{v_1,\cdots, v_p\}$. 
We consider the product map
\beq \label{eq:Stab.quotient.map2}\prod_{j=1}^p\varphi_{v_j}:~\Stab_\tau(T)\lra \prod_{j=1}^p \symS_{I_{v_j}}=\symS_{\lambda(T/\tau)}, \eeq 
where $I$ decomposes as $I=\bigsqcup_{j=1}^p I_j$ with
	\[I_j:=I_{v_j}\sqcup I_{\tau(v_j)}\sqcup \cdots \sqcup I_{\tau^{d_j-1}(v_j)},\]
	for each $j$, where $d_j:=d_{v_j}$.
	This map has the following properties.

\begin{lemma}\label{l:quotient.tree}
	The map \eqref{eq:Stab.quotient.map2} is length-preserving. Moreover, it is surjective and its fibers have the same cardinalities. In particular, for every $\ell>0$,
	\[\frac{\lvert\Stab^\ell_\tau(T)\rvert}{\lvert\Stab(T)\rvert}\leq \frac{\lvert \Stab_e^\ell(T/\tau)\rvert}{\lvert \Stab(T/\tau)\rvert}.\]
\end{lemma}
The inequality will be useful when we prove Theorem~\ref{thm:mu.coeff} in the next subsection, and easily follows from the first two properties, as we will see.

Lemma~\ref{l:quotient.tree} will easily follow from a better understanding of the image of $\Stab_\tau(T)\subset \symS_I$.
	Note that every $\sigma\in \Stab_\tau(T)$ acts on $I_j$ compatibly with the $\tau$-action on the orbit of $v_j$.
	In particular, $\Stab_\tau(T)\subset \symS_I$ factors through 
	\beq\label{eq:Stab.subgroup}\prod_{j=1}^p\left(\symS_{I_{v_j}}\times \symS_{I_{\tau(v_j)}}\times \cdots\times \symS_{I_{\tau^{d_j-1}(v_j)}}\right)\rtimes\langle \tau\rangle \cong \prod_{j=1}^p\left( \symS_{n_j}^{d_j}\rtimes \Z/{d_j}\Z\right)\eeq
	which is a subgroup of $\prod_{j=1}^p \symS_{I_j}\subset \symS_I$, and even further, it factors through
	\beq\label{eq:Stab.subset}
	\prod_{j=1}^p\left(\symS_{I_{v_j}}\times \symS_{I_{\tau(v_j)}}\times \cdots\times \symS_{I_{\tau^{d_j-1}(v_j)}}\right)\rtimes\{ \tau\} \eeq
	as a subset of  \eqref{eq:Stab.subgroup}.
	Here by abuse of notation, we denote any lift in $\symS_{I_j}$ of $\tau$ (restricted to the $\tau$-orbit of $v_j$) again by $\tau$. Note that we may pick any such lift, since any two lifts differ only by reordering labels in each $I_{\tau^i(v_j)}$. 
It is now clear that $\Stab_\tau(T)$ is equal to \eqref{eq:Stab.subset}, since
both have the same cardinalities 
	$\lvert \Stab_\tau(T)\rvert=\lvert \symS_{\lambda(T)}\rvert=\prod_{j=1}^p\lvert \symS_{n_j}^{d_j}\rvert$.

    The above discussion is illustrated in the following example. 

\begin{example}
	Let $T$, $\bar T$, $I$ and $I_{v_j}$ with $0\leq j\leq 3$ be as in Examples~\ref{ex:wrt.trileg} and~\ref{ex:tau}.
	The element $\tau=(1,2,3)\in \Aut(\bar T)=\symS_3$ 
	acts on $v_j$ by $v_1\mapsto v_2\mapsto v_3$. The lift 
	of $\tau$ in $\Stab(T)\subset\symS_{I}$ via the above splitting is then 
	\beq\label{eq:tau.ex}
	\tau=(1,2,3)(4,5,6)(7,8,9) ~\in ~\symS_{I}= \symS_9,\eeq
	which we denote also by $\tau$ by abuse of notation. 
	It follows from this that
	\[\Stab_\tau(T)=\left(\symS_{I_{v_1}}\times \symS_{I_{v_2}}\times \symS_{I_{v_3}}\right)\cdot\tau ~\subset~ \Stab(T)~\subset~ \symS_{I}=\symS_9 \]
	as described in \eqref{eq:Stab.subset}, where $p=1$ in this case.
	
	Moreover, the map \eqref{eq:Stab.quotient.map2} is described as follows. Let $\sigma \in \Stab_\tau(T)$. Express $\sigma:=(\sigma_1,\sigma_2,\sigma_3;\tau)$ as an element in \eqref{eq:Stab.subset}. Note that $\sigma_j:=\sigma_{v_j} \in \symS_{I_{v_j}}$ is chosen so that $\sigma(i)=\sigma_{\tau(v_j)}(\tau(i))$ if $i\in I_{v_j}$. 
    By taking its third power, we obtain
	\[\begin{split}
		\sigma^3&=\big(\sigma_1,~\sigma_2,~\sigma_3;~\tau\big)\cdot \big(\sigma_1,~\sigma_2,~\sigma_3;~\tau\big)\cdot \big(\sigma_1,~\sigma_2,~\sigma_3;~\tau\big)\\
		&=\big(\sigma_1\varphi_\tau(\sigma_3),~\sigma_2\varphi_\tau(\sigma_1),~\sigma_3\varphi_\tau(\sigma_2);~\tau^2\big)\cdot \big(\sigma_1,~\sigma_2,~\sigma_3;~\tau\big)\\
		&=\big(\sigma_1\varphi_\tau(\sigma_3)\varphi_{\tau^2}(\sigma_2),~\sigma_2\varphi_\tau(\sigma_1)\varphi_{\tau^2}(\sigma_3),~\sigma_3\varphi_\tau(\sigma_2)\varphi_{\tau^2}(\sigma_1);~e\big)\\
		&=\big(\sigma_1\tau\sigma_3\tau\sigma_2\tau,~\sigma_2\tau\sigma_1\tau\sigma_3\tau,~\sigma_3\tau\sigma_2\tau\sigma_1\tau;~e\big)
	\end{split}
	\]
	in $\Stab(T)$, where $\varphi_\tau(\sigma_j):=\tau\sigma_j\tau^{-1}\in \symS_{I_{v_{\tau(j)}}}$. Thus, $\sigma^3\in \symS_{I_{v_1}}\times \symS_{I_{v_2}}\times \symS_{I_{v_3}}$. The map \eqref{eq:Stab.quotient.map2} is obtained by composing this with the first projection map
	\beq\label{eq:Stab.quotient.map.ex}\Stab_\tau(T)\xrightarrow{~(-)^3~}\symS_{I_{v_1}}\times \symS_{I_{v_2}}\times \symS_{I_{v_3}}\xrightarrow{~pr_1~}\symS_{I_{v_1}}\eeq
	that sends $(\sigma_1,\sigma_2,\sigma_3;\tau)$ to $\sigma_1\tau \sigma _3\tau\sigma_2\tau \in \symS_{I_{v_1}}$.
	
	For example, when $\sigma=(1,5,9,4,2,6,7,8,3)\in \Stab_\tau(T)\subset \symS_{I}=\symS_9$, it can be written as $\sigma=(\sigma_1,\sigma_2,\sigma_3;\tau)$ with $\sigma_1=(4,7)$, $\sigma_2=(2,5)$, $\sigma_3=(3,6,9)$ and $\tau$ as in \eqref{eq:tau.ex}. The image of $\sigma$ under \eqref{eq:Stab.quotient.map.ex} is the first projection of $\sigma^3=((1,4,7),(2,8,5),(3,9,6))$, which is $(1,4,7)=\sigma_1\tau \sigma_3\tau\sigma_2\tau\in \symS_{I_{v_1}}$. 
	
	 One can also see that composing the third power map with the second or the third projection, rather than the first projection, yields maps, which are conjugate to each other. 
\end{example}
\smallskip

Now we prove Lemma~\ref{l:quotient.tree}.	
\begin{proof}[Proof of Lemma~\ref{l:quotient.tree}]
    For each $\sigma\in \Stab_\tau(T)$, the length of $\sigma$ is equal to the number of its orbits in $I$ while the length of its image under \eqref{eq:Stab.quotient.map2} is equal to the number of the orbits in $I_{v_1}\sqcup \cdots \sqcup I_{v_p}$.
	Since the $\sigma$-orbits in $I_j$ are in a bijective correspondence with the $\sigma^{d_j}$-orbits in $I_{v_j}$ for every $j$, \textbf{\eqref{eq:Stab.quotient.map2}} is length-preserving. 
	
	For the second property, we will  
	endow an action of $\symS_{\lambda(T/\tau)}= \prod_{j=1}^p \symS_{I_{v_j}}$ on $\Stab_\tau(T)$ so that \eqref{eq:Stab.quotient.map2} is $\symS_{\lambda(T/\tau)}$-equivariant, with respect to the left multiplication on the target $\symS_{\lambda(T/\tau)}$. 
	Since the left multiplication is transitive, 
	this will imply that all the fibers are isomorphic to each other.
	
	By letting $n_j=\lvert I_{v_j}\rvert$,	the map 
	$\varphi_{v_j}$ in \eqref{eq:Stab.quotient.map2} 
	is of the form
	\[\symS_{n_j}^{d_j}\rtimes \{\tau\}\hookrightarrow \symS_{n_j}^{d_j}\rtimes \Z/d\Z\lra \symS_{n_j}^{d_j}\xrightarrow{~pr_1~}\symS_{n_j}\]
	that sends $\sigma:=(\sigma_1,\ldots, \sigma_d;\tau)$ with $\sigma_i\in S_{n_j}$ and $\tau=(1,\ldots, d)\in \Z/d\Z$ to
	\[\begin{split}
		\sigma^d
		&=\left(\sigma_1\varphi_\tau(\sigma_d)\cdots \varphi_{\tau^{d-1}}(\sigma_2),~\ldots, ~\sigma_d\varphi_{\tau(\sigma_{d-1})}\cdots \varphi_{\tau^{d-1}}(\sigma_1)~;~e\right) \\
		&=(\sigma_1\tau \sigma_d\tau \cdots \tau \sigma_2\tau , ~\ldots, \sigma_d\tau \sigma_{d-1}\tau \cdots \tau \sigma_1\tau ~;~e) 
		\quad \in ~\symS_{n_j}^{d_j}
	\end{split}
	\]
	and then to the first entry $\sigma_1\tau \sigma_d\tau \cdots \tau \sigma_2\tau$. Here, the $j$-th entry of $\sigma^d$ is 
	\[\sigma_j\varphi_\tau (\sigma_{\tau^{-1}(j)})\cdots \varphi_{\tau^{d-1}}(\sigma_{\tau^{-(d-1)}(j)})=\sigma_j\tau \sigma_{\tau^{-1}(j)}\tau\cdots \tau \sigma_{\tau^{-(d-1)}(j)}\tau \quad \in \symS_{I_{v_{\tau(j)}}}, \]
	and we set $\varphi_\tau(a_j):=\tau a_j \tau^{-1}\in \symS_{I_{v_{\tau(j)}}}$. 
	Thus, this map is $\symS_{n_j}$-equivariant with respect to the left multiplication (only) on the first entry $\sigma_1$. In particular, \eqref{eq:Stab.quotient.map2} is $\prod_{j=1}^p\symS_{I_{v_j}}$-equivariant.
	
	Finally, the last inequality in the assertion follows from  
	\[\frac{\lvert \Stab_\tau^\ell(T)\rvert}{\lvert\Stab_e^\ell(T/\tau)\rvert}= \frac{\lvert \Stab_\tau(T) \rvert}{\lvert \Stab_e(T/\tau)\rvert}= \frac{\lvert \symS_{\lambda(T)}\rvert}{\lvert \symS_{\lambda(T/\tau)}\rvert}\leq \frac{\lvert\Stab(T)\rvert}{\lvert\Stab(T/\tau)\rvert}\]
	where the last inequality follows from Lemma~\ref{l:Stab.decomp}.
\end{proof}

\subsection{Proof of Theorem~\ref{thm:mu.coeff}}\label{S7.4}	
	Let us complete the proof. Recall that
	\[
    \charQ_{n,\bar T}^{n-j}=\sum_{T\in F_n^{-1}(\bar T)}\frac{\lvert \Stab_e^{n-j}(T)\rvert}{\lvert \Stab(T)\rvert} + \sum_{\tau\neq e \in \Aut(\bar T)}\sum_{T\in F_n^{-1}(\bar T)}\frac{\lvert \Stab_\tau^{n-j}(T)\rvert}{\lvert \Stab(T)\rvert}\]
	where the asymptotic formula for the first summand in the right hand side is given in Proposition~\ref{p:summand.identity}.
	When $\tau\neq e$, by Lemma~\ref{l:quotient.tree}, we have
	\[
	0\leq \sum_{T\in F_n^{-1}(\bar T)}\frac{\lvert\Stab^{n-j}_\tau(T)\rvert}{\lvert\Stab(T)\rvert}\leq \sum_{T\in F_n^{-1}(\bar T)}\frac{\lvert \Stab_e^{n-j}(T/\tau)\rvert}{\lvert \Stab(T/\tau)\rvert}.\]
	Moreover, whenever the summand is nonzero on the right hand side, the number $n'$ of inputs of $T/\tau$ is $\lvert\lambda(T/\tau)\rvert$, thus greater than or equal to $n-j$.
	This shows in particular that the right hand side is bounded above by
	\[\sum_{n-j\leq n'<n}\sum_{T'}\frac{\lvert\Stab^{n'-j'}_e(T')\rvert}{\lvert\Stab(T')\rvert}\]
	where $T'$ runs over all the trees obtained by attaching $n'$ inputs to the vertices of $\bar T/\tau$ in the way that they do not necessarily satisfy the valency conditions, and $j':=j-(n-n')$ so that $n'-j'=n-j$.
	
	Since $\tau\neq e$, we have the strict inequalities $j'<j$ and $\lvert V(\bar T/\tau)\rvert <\lvert V(\bar T)\rvert$. 
	Hence, 
	this is dominated 
	by the summand corresponding to $\tau=e$ studied in Proposition~\ref{p:summand.identity}. 
	Combining all these results, we have
	\beq\label{eq:mu.nk.coeff}
    \charQ_{n,k}^{n-j}
    =\sum_{\substack{\bar T\in \sT_k}}
    \charQ_{n,\bar T}^{n-j}= c_k\cdot A_{j,k}(n) 
	+O\left(\frac{k^nn^{2j}}{n!}\right).\eeq 
	This completes our proof of the asymptotic formula for $\charQ_{n,k}$ in Theorem~\ref{thm:mu.coeff}. 
	
	\smallskip
	
	We prove the  case of $\charP_{n,k}$.
	First note that by definition,
	\[\sum_{\substack{n_1+n_2=n,\\j_1+j_2=j}}A_{j_1,k_1}(n_1)\cdot A_{j_2,k_2}(n_2)=A_{j,k_1+k_2+1}(n)\]
	for any $k_1,k_2$ and $j\geq0$. Therefore, for $k\geq0$,
	\[\begin{split}
		&\lim_{n\to \infty}\frac{n!}{(k+1)^{n-j}c(n,n-j)}\sum_{\substack{n_1+n_2=n,\\j_1+j_2=j}}A_{j_1,k_1}(n_1)\cdot A_{j_2,k_2}(n_2)\\
		&=\lim_{n\to \infty}\frac{n!}{(k+1)^{n-j}c(n,n-j)}\cdot A_{j,k_1+k_2+1}(n)\\
	&= \lim_{n\to \infty}\left(\frac{k_1+k_2+2}{k+1}\right)^{n-j}=\begin{cases}
			1 &\text{ if }~k_1+k_2=k-1\\
			0 &\text{ if }~k_1+k_2<k-1.
		\end{cases}
	\end{split}
	\]
    Now read the coefficients of $m^{n-j}$ off \eqref{eq:walcrossing.charpoly},
    multiply them by $\frac{n!}{(k+1)^{n-j}c(n,n-j)}$ and take their limits as $n\to \infty$. Then by \eqref{eq:mu.nk.coeff} and the above observation, 
	\[\lim_{n\to\infty}\frac{n!\cdot \left(\charP_{n,k}^{n-j}+\charP_{n,k-1}^{n-j}\right)}{(k+1)^{n-j}c(n,n-j)}\]
	is equal to 
	\[
	\begin{split}
		& c_k-\frac{1}{2}\lim_{n\to \infty}\sum_{\substack{j_1+j_2=j,\\n_1+n_2=n,\\k_1+k_2=k-1}}\frac{n!\cdot 
        \charQ_{n_1,k_1}^{n_1-j_1}\cdot 
        \charQ_{n_2,k_2}^{n_2-j_2}}{(k+1)^{n-j}c(n,n-j)}\\
		&= c_k -\frac{1}{2}\lim_{n\to \infty}\sum_{\substack{j_1+j_2=j,\\n_1+n_2=n,\\k_1+k_2=k-1}}\frac{n!\cdot c_{k_1}\cdot A_{j_1,k_1}(n_1)\cdot c_{k_2}\cdot A_{j_2,k_2}(n_2)}{(k+1)^{n-j}c(n,n-j)} \\
		&= c_k -\frac{1}{2}\lim_{n\to \infty}\sum_{\substack{k_1+k_2=k-1}}\frac{c_{k_1}c_{k_2}\cdot n!\cdot A_{j,k}(n)}{(k+1)^{n-j}c(n,n-j)} =c_k-\frac{1}{2}\sum_{k_1+k_2=k-1}c_{k_1}c_{k_2},
	\end{split}
	\]
	which is equal to $d_k$.
	This completes our proof of Theorem~\ref{thm:mu.coeff}.
	\hfill $\square$

\bibliographystyle{amsplain}

\end{document}